\documentclass{article}
\usepackage{fullpage}
\usepackage{amsmath,amssymb,enumerate}
\usepackage{url}

\usepackage{tikz}

\usepackage{amsthm}

\newtheorem{theorem}{Theorem}[section]

\newtheorem{lemma}[theorem]{Lemma}

\theoremstyle{definition}
\newtheorem{definition}[theorem]{Definition}
\newtheorem{remark}[theorem]{Remark}

\newtheorem{example}{Example}[section]

\mathchardef\ordinarycolon\mathcode`\:
\mathcode`\:=\string"8000\def\R{\mathbb{R}}
\begingroup \catcode`\:=\active
  \gdef:{\mathrel{\mathop\ordinarycolon}}
\endgroup

\renewcommand{\R}{\mathbb{R}}

\newcommand{\rr}{\mathbb{R}}

\usepackage{mathrsfs}
\newcommand{\Ac}{\mathcal{A}}

\newcommand{\Dc}{\mathcal{D}}
\newcommand{\Ec}{\mathcal{E}}
\newcommand{\Fc}{\mathcal{F}}

\newcommand{\Lc}{\mathcal{L}}
\newcommand{\Mc}{\mathcal{M}}

\newcommand{\Qc}{\mathcal{Q}}

\newcommand{\Sc}{\mathcal{S}}

\newcommand{\WF}{\mathrm{WF}}                         % Wavefront set
\newcommand{\norm}[1]{\left\lVert#1\right\rVert}      % Norm
\newcommand{\abs}[1]{\left|#1\right|}                 % Absolutbetrag
\newcommand{\paren}[1]{\left(#1\right)}               % Klammern
\newcommand{\bparen}[1]{\left[#1\right]}               % eckige Klammern
\newcommand{\sparen}[1]{\left\{#1\right\}}		      % Mengenklammer
\renewcommand{\d}{\,\mathrm{d}}						  % within the integral sign

\DeclareMathOperator*{\supp}{\mathrm{supp}}           % support
\DeclareMathOperator*{\singsupp}{\mathrm{sing\:supp}} % singular support
\renewcommand{\epsilon}{\varepsilon}
\renewcommand{\rho}{\varrho}
\newcommand{\vp}{\varphi}
\newcommand{\thperp}{{\theta^\bot}}

\renewcommand{\tilde}{\widetilde}
\newcommand{\rtwo}{{{\mathbb R}^2}}
\newcommand{\rthree}{{{\mathbb R}^3}}
\newcommand{\rn}{{{\mathbb R}^n}}

\newcommand{\eps}{\varepsilon}

\newcommand{\st}{\hskip 0.3mm : \hskip 0.3mm}
\renewcommand{\th}{\theta}

\newcommand{\be}{\begin{equation}}
\newcommand{\ee}{\end{equation}}
\newcommand{\bea}{\begin{eqnarray}}
\newcommand{\eea}{\end{eqnarray}}
\newcommand{\bean}{\begin{eqnarray*}}
\newcommand{\eean}{\end{eqnarray*}}

\newcommand{\bel}[1]{\begin{equation}\label{#1}}

\newcommand{\eel}[1]{{\label{#1}\end{equation}}}

\newcommand{\inv}{^{-1}}

\newcommand{\intt}{{\operatorname{int}}}
\newcommand{\range}{{\operatorname{range}}}

\newcommand{\bd}{{\operatorname{bd}}}

\newcommand{\smo}{\setminus\boldsymbol{0}}

\newcommand{\xio}{{\xi_0}}
\newcommand{\xo}{{x_0}}

\newcommand{\PD}{\boldsymbol{d}}
\newcommand{\A}{\alpha}
\newcommand{\B}{\beta}
\newcommand{\g}{\gamma}

\newcommand{\chiab}{\chi_{\mathrm{[a,b]}}}
\newcommand{\chia}{\chi_{\mathrm{A}}}
\newcommand{\dx}{\mathbf{dx}}
\newcommand{\dy}{\mathbf{dy}}

\newcommand{\dr}{\mathbf{dr}}
\newcommand{\ds}{\mathbf{ds}}

\newcommand{\dphi}{\mathbf{d\phi}}
\newcommand{\Om}{\Omega}

\newcommand{\Mab}{\mathcal{M}_{[a,b]}}
\newcommand{\Cmt}{C^t_\mathcal{M}}
\newcommand{\Cm}{C_\mathcal{M}}
\newcommand{\Vc}{\mathcal{V}}
\newcommand{\Vab}{\mathcal{V}_{[a,b]}}
\newcommand{\xoxio}{(x_0,\xi_0\dx)}

\newcommand{\sxxi}{\left\{(x,\xi\dx)\right\}}

\newcommand{\Aab}{\mathcal{A}_{\{a,b\}}}
\newcommand{\Wab}{W_{\{a,b\}}}

\newcommand{\Kvp}{\mathcal{K}_\varphi}

\newcommand{\n}{\mathbf{n}}
\newcommand{\npx}{\mathbf{n}(\phi,x)}

\newcommand{\Lvp}{\mathcal{L}_{\varphi}}

\newcommand{\Ma}{\mathcal{M}_\mathrm{A}}
\newcommand{\Mst}{\mathcal{M}^\ast}

\newcommand{\nyx}{\mathbf{n}(y,x)}

\newcommand{\Vs}[1]{\mathcal{V}_{\mathcal{S},#1}}
\newcommand{\As}[1]{\mathcal{A}_{\mathcal{S},#1}}

\newcommand{\Vsk}{\mathcal{V}_{\mathcal{S},K}}

\newcommand{\Ask}{\mathcal{A}_{\mathcal{S},K}}

\newcommand{\Ms}{\mathcal{M}_{\mathcal{S}}}
\newcommand{\Msst}{\mathcal{M}_{\mathcal{S}}^*}
\newcommand{\Msa}{\mathcal{M}_{\mathcal{S},A}}

\newcommand{\Cs}{C_\mathcal{S}}
\newcommand{\Cst}{C^t_\mathcal{S}}

\newcommand{\Xis}{\Xi_\mathcal{S}}

\newcommand{\tx}{\tilde{x}}
\newcommand{\om}{\omega}

\newcommand{\Wbd}{W_{A}}

\newcommand{\alphao}{\alpha_0}
\newcommand{\etao}{\eta_0}

\newcommand{\po}{\phi_1}
\newcommand{\pt}{\phi_2}

\newcommand{\cinv}{c_\text{inv}}
\newcommand{\vn}{\overline{n}}

\usepackage[babel]{csquotes}
\usepackage{longtable}

\title{Artifacts in incomplete data tomography\\ with applications to photoacoustic tomography and sonar}

\author{J\"urgen Frikel\footnotemark[1] \and  Eric Todd 
Quinto\footnotemark[2]%\newline\textbf{Email:} todd.quinto@tufts.edu}
}

\begin{document}
\date{}
\maketitle
%\slugger{mms}{xxxx}{xx}{x}{x--x}%slugger should be set to mms, siap, sicomp, sicon, sidma, sima, simax, sinum, siopt, sisc, or sirev

\renewcommand{\thefootnote}{\fnsymbol{footnote}}
\footnotetext[1]{Department of Mathematics, Tufts University,
Medford, MA 02155, USA and Institute of Computational Biology,
Helmholtz Zentrum M\"unchen, Germany, \textbf{Email:}~\texttt{juergen.frikel@helmholtz-muenchen.de}}
\footnotetext[2]{Department of Mathematics, Tufts University,
Medford, MA 02155, USA; \textbf{Email:}~\texttt{todd.quinto@tufts.edu}}

\begin{abstract}
We develop a paradigm using microlocal analysis that allows one to
characterize the visible and added singularities in a broad range of
incomplete data tomography problems. We give precise characterizations
for photo- and thermoacoustic tomography and Sonar, and provide
artifact reduction strategies. In particular, our theorems show that
it is better to arrange Sonar detectors so that the boundary of the
set of detectors does not have corners and is smooth. To illustrate
our results, we provide reconstructions from synthetic spherical mean
data as well as from experimental photoacoustic data.

%We apply it to thermoacoustic
%tomography and sonar, and we provide reconstructions from real and
%simulated data that illustrate our theorems.  In particular, our
%theorems show that it is better to arrange sonar detectors so that the
%boundary of the set of detectors does not have corners and is smooth.
\end{abstract}

% \begin{keywords}
% Computed Tomography, Photoacoustic tomography, Thermoacoustic
% tomography, Sonar, Lambda Tomography, Limited Angle Tomography, Radon
% transforms, Spherical mean, Microlocal Analysis, Fourier integral
% operators.
% \end{keywords}
% 
% \begin{AMS}Primary: 44A12, 92C55, 35S30 Secondary: 65R10, 58J40 \end{AMS}

\pagestyle{myheadings}
\thispagestyle{plain}
%\markboth{TEX PRODUCTION}{USING SIAM'S \LaTeX\ MACROS}

%\jc{Decide on an appropriate title}

%\jc{SIAM Journal on Applied Mathematics has a page limit policy of 20 pages per paper. This limit can be exceeded; however, substantial deviations require that the referees, editor, and editor-in-chief be convinced that the increased length is both required by the subject matter and justified by the quality of the paper.}

\section{Introduction}

In many types of computed tomography, such as x-ray tomography,
photoacoustic (and thermoacoustic) tomography (PAT/TAT) or Sonar, the
tomographic projections can be acquired only from a limited field of
view. As a result, the data are highly incomplete and the
corresponding reconstruction problem becomes severely ill-posed which
leads to serious instabilities of the reconstruction process. As a
consequence two phenomena can be observed in practical
reconstructions: First, only specific features of the unknown object
(visible singularities) can be reconstructed reliably, cf.
\cite{Palamodov:JFAA,Quinto93} and Figure \ref{fig:pat circ rec}.
Second, and even more important, additional singularities (artifacts)
can be generated during the reconstruction and superimpose reliable
information, cf.  \cite{FrikelQuinto2013,Katsevich:1997} and Figure
\ref{fig:pat circ rec}. This is a serious problem, since artifacts can
overlap and generate new image features leading to misinterpretations
or possibly misdiagnosis in medical imaging applications. It is
therefore essential to develop a precise understanding of such
artifacts for a range of imaging situations and to provide algorithms
that reliably reconstruct the information that is contained in the
data and at the same time avoid the generation of unwanted features.

The generation of artifacts in incomplete data PAT and Sonar has been
addressed in several publications, eg.
\cite{Buehler:2011bv,Haltmeier:2009ee,Jaeger:2007fy,Natterer86,Nguyen:2013et,Patrickeyev:2004oa,Xu:2004fe},
to mention only a few. In particular, it has been observed that those
artifacts occur due to hard truncation of the data.  In order to
reduce the generation of artifacts in practical 
reconstructions, some authors therefore use smooth truncation of the
limited view data. Although it is intuitively understood why artifacts
occur and how to deal with them, to the best of our knowledge, no
theoretical (geometrical) characterization of artifacts was given so
far - neither in PAT nor in Sonar.  Also, it has not yet been
mathematically justified, in general, why smooth truncation
reduces artifacts (see Remark \ref{remark:subtle}).

In this article, we use the framework of microlocal analysis and the
calculus of Fourier integral operators to develop a general approach
that enables one to mathematically characterize limited angle
artifacts for different types of tomography problems. We show that the
reason for artifact generation is the hard truncation of the data at
the ends of the angular range, and that they can be reduced by using a
smooth truncation.  We provide a paradigm that applies to a broad
range of limited data problems and derive explicit characterizations of 
artifacts PAT/TAT and Sonar. Moreover, we illustrate our results in
numerical experiments on simulated and experimental PAT data. 

Characterizations of limited angle artifacts for x-ray tomography were
obtained in \cite{Katsevich:1997} and \cite{FrikelQuinto2013}, where
also a smooth truncation of the limited data was proposed to reduce
the artifacts.  Katsevich's results \cite{Katsevich:1997} apply to the
line transform with arbitrary smooth weights.  The results in
\cite{FrikelQuinto2013,Katsevich:1997} have been derived in a way
which does not directly generalize to other tomography problems.  This
is mainly due to the fact that the authors heavily rely on the
explicit expression of the reconstruction operators as singular
pseudodifferential operators.  However, for many other tomography
problems such formulas are not available and, hence, the techniques of
\cite{FrikelQuinto2013,Katsevich:1997} cannot be applied in order to
get similar characterizations.  Nguyen has qualitatively
analyzed the strength of the added artifacts occurring in limited
angle tomography \cite{Nguyen:2014}.  His recent article
\cite{Nguyen:2014b} uses ideas from this article plus microlocal,
asymptotic arguments to calculate the strength of the artifacts for
the spherical transform in cases related to the ones we discuss here. In
\cite{ParkChoiSeo:2014} microlocal analysis is used to understand the
streaks in X-ray CT scans caused by metal.  In
\cite{SymesSeismics1998}, Symes observed the advantages of soft
truncation (p.\ 46-47) and the problems with hard truncation (p.\ 65)
in seismic imaging.

\begin{figure}[t]
\centering
	\includegraphics[width=4cm]{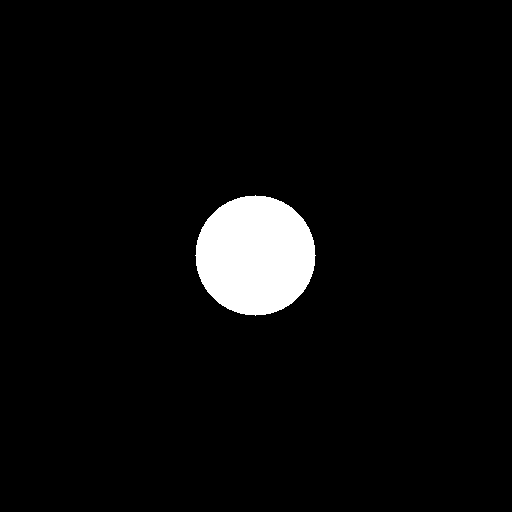}\hskip2ex
%\subfloat[Original]{} circ_lambda_25-155deg_CE_70-235
% 	\includegraphics[width=4cm]{./img/circ_lambda_0-180deg_CE_70-235.png}\hskip2ex %\subfloat[$\Lc^2 g$ - no artifact reduction]{}
	\includegraphics[width=4cm]{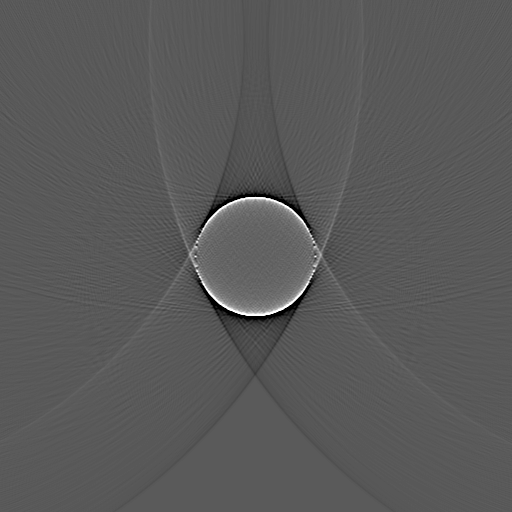}\hskip2ex
	\includegraphics[width=4cm]{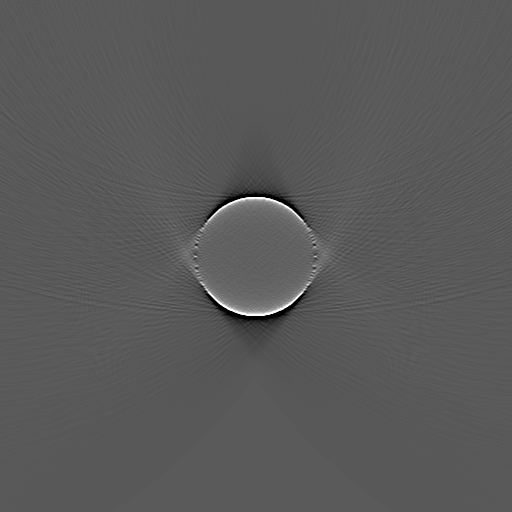}
	\caption{Lambda type reconstruction from limited view spherical mean data for the limited angular range $[25^\circ,155^\circ]$ ($361$ projections, $725$ radii). %, $\eps=18^\circ$ 
	\emph{Left:} original image ($512\times 512$, characteristic function of a disc centered at the origin); \emph{Middle:} Lambda reconstruction without artifact reduction; \emph{Right:} Lambda reconstruction with artifact reduction.   %$\eps=0.1\!\cdot\!\pi$
%	\emph{Left:} Without artifact reduction. \emph{Middle:} With artifact reduction. \emph{Right:} Difference image. In above reconstructions, one can clearly observe the effect of artifact reduction.
}
\label{fig:pat circ rec}
\end{figure}

This article is organized as follows. In Section \ref{sect:microlocal analysis} we introduce the basic
microlocal analysis needed for the article, including fundamental
theorems about Fourier integral operators. Then, in Section
\ref{sect:general strategy}, we give the key Theorem \ref{thm:WF
mult}, describe the general strategy, and outline our paradigm to
characterize the added singularities.  This will guide the proofs and
show how the paradigm can be used in general.  In Section
\ref{sect:characterization}, we apply these results to thermo- and
photoacoustic tomography and Sonar. In section \ref{sect:numerics}
we provide reconstructions from real data to show how the artifacts
occur and how they can be decreased.  Finally, in section
\ref{sect:conclusions}, we make some general observations about our
method. Proofs of our main theorems are provided in the Appendix.

\section{Microlocal Analysis and Fourier Integral Operators}
\label{sect:microlocal analysis}

In this section we review basic facts from microlocal analysis and the
calculus of Fourier integral operators (including fundamental
theorems) needed for the article.  For general facts about the theory
of distributions and more details on microlocal analysis we refer to
\cite{Friedlander98,Hoermander03}. For details on Fourier integral
operators we refer to \cite{Ho1971,Treves:1980vf}.
 
Let $\Om$ be an open set.  We denote the set of $C^\infty$ functions
with domain $\Om$, by $\Ec(\Om)$ and the set of $C^\infty$ functions
of compact support in $\Om$ by $\Dc(\Om)$.  Distributions are
continuous linear functionals on these function spaces.  The dual
space to $\Dc(\Om)$ is denoted $\Dc'(\Om)$ and the dual space to
$\Ec(\Om)$ is denoted $\Ec'(\Om)$. In fact, $\Ec'(\Om)$ is the set of
distributions of compact support in $\Om$.  For more information about
these spaces we refer to \cite{Rudin:FA}.

A function $f(\xi)$ is said to \emph{decay rapidly at infinity in a
conic open set $V$} if it decays faster than any power of
$1/\norm{\xi}$ in $V$.  The singular support of a distribution, $f$,
$\singsupp(f)$, is the complement of the largest open set on which $f$
is a $C^\infty$ function.  It follows directly from this definition
that $\singsupp(f)\subset\supp(f)$, and $\singsupp(f)=\emptyset$ if
and only if $f\in C^\infty(\rn)$.

To make the concept of singularity apply to our range of problems, we
will need to view the wavefront set as a subset of a conormal bundle
so it will be invariantly defined on manifolds \cite{Treves:1980vf}.
If $\Xi$ is a manifold and $y\in \Xi$, then the cotangent space of
$\Xi$ at $y$ is the set of all first order differentials (the dual
space to the tangent space $T_y(\Xi)$), and the cotangent bundle
$T^*(\Xi)$ is the vector bundle with fiber above $y\in \Xi$.  That is $T^*(\Xi)=\sparen{(y,\eta)\st y\in \Xi, \eta\in
T^*_y(\Xi)}$.

For example, the differentials $\dx_1$, $\dx_2,\dots$, and $\dx_n$ are
a basis of $T^*_x(\rn)$ for any $x\in \rn$.  For $\xi\in \rn$, we will
use the notation \[\xi\dx = \xi_1\dx_1+\xi_2\dx_2+\cdots+\xi_n\dx_n\in
T^*_x(\rn).\]  If $\phi\in \rr$ then $\dphi$ will be the differential
with respect to $\phi$ and $\dr$ and $\ds$ are defined analogously.

If $\phi$ is a function of $x,y$ then $\PD_x \phi$ is the differential
of $\phi$ in $x$, so if $\nabla_x \phi$ is the gradient of $\phi$ in
$x$, then \bel{def:PD}\PD_x\phi = \nabla_x \phi \dx.\ee The
differentials of $\phi$ in other variables will be defined in a
similar way.

\begin{definition}[Wavefront Set
{\cite{Hoermander03}}]\index{wavefront set}\label{def:wavefront set}
Let $f\in\Dc'(\R^n)$, $\xo\in \rn$ and $\xio\in \rn\smo$.  Then \emph{$f$
is microlocally smooth at $\xo$ in direction $\xio$} if there is a cutoff
function $\vp$ (a smooth function of compact support for which
$\vp(\xo)\neq 0$) and a conic neighborhood of $\xio$ such that the
localized Fourier transform  $\widehat{(\vp
f)}$ is rapidly decreasing at infinity in $V$.  

The \emph{wavefront set} of $f$ is the set $\WF(f)$ of all $(x,\xi\dx)\in
T^*(\rn)\smo$ such that $f$ \emph{is not} microlocally smooth at $x$
in direction $\xi$.
\end{definition}

\begin{example}\label{ex:char}
Let $K$ be a compact subset of $\rtwo$ bounded by a simple closed
smooth curve $B=\bd(\Omega)$. Then, the wavefront set of the
characteristic function of $K$, $\chi_\Omega$, is the set of
covectors conormal to the boundary of $K$:
\begin{equation} \label{eq:wavefront set characteristic function}
(y,\eta\dy)\in \WF(\chi_K)\quad\Leftrightarrow\quad
y\in\bd(K),\;\eta\in N_y,\; \end{equation} where $N_y$ is the set of
all vectors normal to $\bd(K)$ at $y$.  The proof of this fact is
non-trivial.  The wavefront set of the characteristic function of a
square is the set of conormals to the sides plus all covectors above
the corners of the square (see \cite[Example 9.12, p.\
219]{Petersen}).  This can be used to show \eqref{eq:wavefront set
characteristic function} by using the Inverse Function Theorem to find
a diffeomorphism to locally straighten out the boundary curve.  Then
note that the diffeomorphism takes conormals to one boundary (i.e.,
wavefront set) to conormals of the other.

However, if $\bd(K)$ is piecewise smooth and has a corner at a point
$y$ then all covectors above $y$ are in $\WF(\chi_K)$.  This follows
from the same example in \cite{Petersen} or a Radon line transform
argument.  
\end{example}

We now introduce Fourier Integral Operators (FIO) and provide some of
their properties. These operators are generalizations of differential
operators and they alter wavefront sets in precise ways.

\begin{definition}[{\cite{Treves:1980vf}}]\label{Phase Fcn FIO}
Let $Y\subset \rr^{m}$ and $X\subset \rr^{n}$ be open subsets. A real
valued function $\phi\in C^{\infty}(Y\times X\times
\rr^{N}\setminus\{0\})$ is called a \emph{phase function} if 
\begin{enumerate}
\item $\phi$ is positive-homogeneous of degree $1$ in $\xi$.  That is
$\phi(y,x,r\xi)=r\phi(x,y,\xi)$ for all $r>0$.  \item
$(\PD_{y}\phi,\PD_{\xi}\phi)$ and $(\PD_{x}\phi,\PD_{\xi}\phi)$ do not
vanish for all $(y,x,\xi)\in Y\times X\times\rr^{n}\setminus\{0\}$
where $\PD_{x}$ is defined in \eqref{def:PD} and the other operators
are defined in a similar way. 
\end{enumerate}
We define the auxiliary manifold
\bel{def:aux.manifold}
\Sigma_\phi = \sparen{(y,x,\xi)\in Y\times X\times \paren{\rn\smo}\st
\PD_{\xi}\phi(y,x,\xi) = 0}.
\ee

The phase function $\phi$ is called non-degenerate if the set
$\{\PD_{y,x,\xi}\paren{\frac{\partial \phi}{\partial \xi_{j}}}$,
$1\leq j\leq N\}$ is linearly independent on $\Sigma_\phi$.
\end{definition}

\begin{definition}[{\cite{Treves:1980vf}}]\label{DEF3:FIO}
A Fourier integral operator (FIO) $\Fc$ is defined as 
\[ \Fc u(y)=\int e^{i\phi(y,x,\xi)}p(y,x,\xi)u(x)\,d x \,d \xi,\]
where $\phi$ is a non-degenerate phase function, and the amplitude $p(y,x,\xi)\in C^{\infty}(Y\times X\times
\rr^{n})$ satisfies the following estimate: For every compact
set $K\subset Y\times X$ and for every multi-index $\A,\B,\g$, there
is a constant $C=C(K,\A,\B,\g)$ such that 
\[ \abs{D_{\xi}^{\A}D_{x}^{\B}D_{y}^{\g}p(y,x,\xi)}
\leq C (1+\norm{\xi})^{m-\abs{\A}} \mbox{ for all } x,y\in K \mbox{ and
for all } \xi\in \rr^{n}.\]

The \emph{canonical relation} of $\Fc$ is defined as
\bel{def:canonical relation} C:=\sparen{\paren{y,\PD_y \phi(y,x,\xi);
x, -\PD_x \phi(y,x,\xi)}: (y,x,\xi)\in \Sigma_{\phi}}.  \ee
\end{definition}%

Note that since the phase function $\phi$ is non-degenerate, the sets
$\Sigma_{\phi}$ and $C$ are smooth manifolds.  Furthermore, $C$ is
conic in the cotangent variables

To understand what Fourier integral operators and their compositions
do on wavefront sets, we will define compositions of canonical
relations.  Let $X$ and $Y$ be manifolds, and $A\subset T^\ast (X)\times
T^\ast (Y)$, then 
\begin{align}
	A' &= \sparen{(x,\xi;y,-\eta)\st  (x,\xi; y,\eta)\in A},\notag\\
	A^t &= \sparen{(y,\eta;x,\xi)\st  (x,\xi;y,\eta)\in
	A}.\label{def:At}
\end{align}
If $B\subset T^\ast (Y)\times T^*(X)$ and $C\subset T^*(X) $, we define
\begin{equation*}
	B\circ C = \sparen{(y,\eta)\in T^\ast (Y)\st  \exists (x,\xi)\in
	C\st (y,\eta;x,\xi)\in B},
\end{equation*}
and 
\begin{multline}
	A\circ B = \{(x,\xi; x',\xi')\in T^\ast (X)\times T^\ast
	(X)\st  \\ \exists (y,\eta)\in T^\ast (Y)\st  (x,\xi; y,\eta)\in A\text{ and }( y,\eta; x',\xi')\in B\}.
\end{multline}
For later use, we note the following relations: For $A$, $B$, and $C$
as above and $\tilde{C}\subset T^*(X)$
\begin{equation}
\label{eq:distributive and associative relations for compositions}
	B\circ \big(C\cup \tilde{C}\big) = \big(B\circ C\big) \cup
	\big(B\circ \tilde{C}\big),\qquad  A\circ \big(B\circ C\big) =\big(A\circ B \big)\circ C.
\end{equation}

Note that a linear operator, $L:\Ec'(X)\to \Dc'(Y)$ is \emph{properly
supported} when the following holds: if $S$ is the support of the
Schwartz kernel of $L$, then the projections from $S$ to $X$ and to
$Y$ are compact maps (i.e., the inverse image of any compact set is
compact).  This implies that $L:\Ec'(X)\to \Ec'(Y)$.  Now we make use
of the fact that Fourier integral operators satisfy the H\"ormander
Sato Lemma. 

\begin{theorem}[Th.\ 5.4, p.\
461 {\cite{Treves:1980vf}}]\label{thm:HS} Let $f$ be a distribution of
compact support and let $\Fc$ be a Fourier integral operator.  Then, 
\begin{equation}
\label{eq:wavefront set of a FIO}
	\WF(\Fc f)\subset C\circ\WF(f).
\end{equation}
If $\Fc$ is properly supported, then this inclusion is valid for any
distribution.
\end{theorem}

Furthermore we have that the adjoint of a FIO is a FIO.

\begin{theorem}[Thm.\ 4.2.1 p.\ 174
{\cite{Ho1971}}]\label{thm:FIOadjoint} If $\Fc$ is an FIO associated
to the canonical relation $C$, then the adjoint $\Fc^*$ is an FIO
associated to $C^t$.\end{theorem}

These theorems and the composition relations for FIO will be the keys
to our general strategy in the next section and the proofs in the
subsequent sections.

\section{General Strategy}
\label{sect:general strategy} 

In this section, we will outline the general ideas we will apply in
the following sections to understand visible and added singularities
in limited data tomography. By presenting the ideas in general, we
emphasize the broad applicability of this mathematics.

The imaging operator will be denoted $\Mc:\Ec'(\Omega)\to\Ec'(\Xi)$,
where the \emph{object space} $\Omega$ is a region in space to be
imaged and the \emph{data space} $\Xi$ is a space that parameterizes
the data.  For the planar X-ray transform, the imaging operator is the
X-ray transform, $\Omega$ is an open set in $\rtwo$ containing the
object to be imaged, and $\Xi$ is the set of lines in $\rtwo$. In what
follows the operator $\Mc$ is assumed to be a FIO.  In this article,
we consider incomplete data problems in which the data are taken only
on a closed set $A\subset \Xi$.  The resulting forward operator can be
written \bel{MA} \Ma f = \chia \Mc,\ee where $\chia$ is the
characteristic function of $A$ and the product just restricts the data
to the set $A$.  In the cases we consider, the reconstruction operator
is of the form \bel{MA-reconstruction}\Mst P\Ma,\ee where $\Mst$ is an
appropriate dual or backprojection operator to $\Mc$ that takes
functions on the data space to functions on the object space and $P$
is a differential or pseudodifferential operator.  Equation
\eqref{MA-reconstruction} models many standard reconstruction
algorithms, including limited angle filtered backprojection
\cite{Natterer86}, Lambda tomography \cite{FFRS,FRS}, and algorithms
in thermoacoustic tomography \cite{Finch-P-R, Kun:sphere} and sonar
\cite{And}, and radar \cite{NC2002}.

Since $\Mc$ is assumed to be a FIO, Theorem \ref{thm:HS} and the
wavefront relation \eqref{eq:wavefront set of a FIO} tells what $\Mc$
and $\Mc^*$ do to $\WF(f)$.  Our next theorem tells what
multiplication by $\chia$ does to the wavefront set.  It is a special case
of Theorem 8.2.10 in \cite{Hoermander03}.

 \begin{theorem}\label{thm:WF mult} Let $u\in \Dc'(\Xi)$, and
let $A$ be a closed subset of $\,\Xi$ with nontrivial interior.  If
the \emph{non-cancellation condition} \bel{non-cancellation}
\forall\,(y,\xi)\in \WF(u),\ 
 (y,-\xi)\notin \WF(\chia)\ee holds, then the product
$\chia u$ can be defined as a distribution.  In this case, we have
\bel{WF of a product}\WF(\chia u) \subset \Qc(A,\WF(u)),\ee
where, for
$W\subset T^*(\Xi)$, \bel{def:Q}\begin{aligned}\Qc(A,W) :=&
\big\{(y,\xi+\eta)\st y\in A\,, \bparen{(y,\xi)\in W\text{\rm\ or }
\xi = 0} \text{\rm\ and } \big[(y,\eta)\in \WF(\chia)\text{\rm\ or }
\eta = 0\big]\big\}\,.\end{aligned}\ee
\end{theorem}

Note that the condition ``$y\in A$'' is not in \eqref{def:Q} in
H\"{o}rmander's theorem, but we can include this condition because
$\chia u$ is zero (hence smooth) off of the closed set $A$.

An auxiliary lemma will make the paradigm easier to apply.

\begin{lemma}\label{lemma:composition identity}  Let  $\Mc:\Ec'(X)\to
\Dc'(\Xi)$ be a FIO with canonical relation $C$ and let $A$ be a
closed subset of $\Xi$.  Assume the non-cancellation condition
\eqref{non-cancellation} holds for $\Mc$ and $\chia$ so the Schwartz
kernel of $\Ma=\chia \Mc$ is a distribution.  Assume the linear
operator $\Ma:\Ec'(X)\to \Ec'(\Xi)$ (i.e., for each $f\in \Ec'(X)$,
the distribution $\Ma(f)$ has compact support).  Let $P$ be a properly
supported pseudodifferential operator (or $\Mst:\Dc'(\Xi)\to
\Dc'(X))$.  Then, 
\begin{gather}\label{composition-genl:Mf}\WF(P \Ma f)\subset 
\Qc(A,C\circ \WF(f)),\\
\label{composition-genl}  
\WF(\Mst P \Ma f)\subset C^t\circ\Qc(A,C\circ \WF(f)).\end{gather}
\end{lemma}

\begin{proof}  By  Theorem \ref{thm:HS}, $\WF(\Mc
f)\subset C\circ \WF(f)$.  Then one uses Theorem \ref{thm:WF mult} and
the definition of $\Qc$, \eqref{def:Q}, to infer 
 \[\WF(\chia\Mc f)\subset \Qc\paren{A,\WF(\Mc f)}
 \subset\Qc\paren{A,C\circ \WF(f)}.\] We also just used the fact that
when $W'\subset W\subset T^*(\Xi)$, then $\Qc(A, W')\subset \Qc(A,W)$.
This finishes the proof of \eqref{composition-genl:Mf}.  Now, using
Theorem \ref{thm:HS} and the composition rules \eqref{eq:distributive
and associative relations for compositions} for $C^t$, one proves
\eqref{composition-genl} from \eqref{composition-genl:Mf}.
\end{proof} 

Here is the outline of our paradigm; it can be used for a range of
limited data problems to understand visible and added singularities.

\begin{enumerate}[(a)]
\item\label{paradigm:start} Confirm the forward operator $\Mc$ is a
FIO and calculate its canonical relation, $C$.

\item Choose a closed limited data set $A\subset \Xi$ and calculate
$\WF(\chia)$ (see Example \ref{ex:char}).

\item Make sure the non-cancellation condition
\eqref{non-cancellation} holds for $\chi_A$ and $\Mc f$.  This can be
done in general by making sure it holds for $(y,\eta)\in C\circ
\paren{T^*(\Om)\smo}$ since that is the largest that $\WF(\Mc f)$ can be (by
Theorem \ref{thm:HS} since $\WF(f)\subset T^*(\Om)\smo $).  

\item \label{paradigm:Mf} Then, calculate $\Qc(A,C\circ \WF(f))$.

\item By Lemma \ref{lemma:composition identity} \bel{final WF
containment}\WF(\Mst P \Ma f)\subset C^t\circ\Qc\paren{A,
C\circ\WF(f)},\ee so calculate $C^t\circ\Qc\paren{A, C\circ\WF(f)}$ to
find possible visible singularities and added artifacts.

\end{enumerate}

\section{Characterization of Limited Data Artifacts in PAT/TAT and
Sonar}
\label{sect:characterization} 

Using the paradigm of section \ref{sect:general strategy}, we now
describe the visible and added singularities for photo- and
thermoacoustic tomography (PAT and TAT, respectively), and sonar with
constant sound speed.  Proofs will be given in the appendix.

The same arguments can be used to prove the theorems in
\cite{FrikelQuinto2013} about limited angle tomography, even for the
generalized X-ray transform in the plane with arbitrary smooth
measures.  The arguments in \cite{FrikelQuinto2013,Katsevich:1997}
are more elementary and do not require the theory of FIO, but they do
not apply to generalized transforms.

\begin{remark}\label{remark:generalized}
Although we state the theorems for the circular and spherical
transforms with standard measures, our theorems are valid more
generally.  Theorems \ref{thm:visible and added sing} and
\ref{thm:reduction:circle} are valid for the generalized circular mean
transform with a smooth nowhere zero weight $\mu(\phi,x)\d x$ in
\eqref{def:Mc:circle}.  Theorems \ref{thm:visible and added sing
sphere} and \ref{thm:reduction:sphere} are valid for the generalized
circular mean transform with a smooth nowhere zero weight $\mu(y,x)\d
x$ in \eqref{def:Ms}.  In fact, our theorems are valid for any FIO
associated to the same canonical relation since our proofs use only
the properties of the operators as FIO and their associated their
canonical relations (when the operators can be composed).
\end{remark}

 \subsection{Photo- and thermoacoustic tomography for planar data
(circular Radon transform)}\label{sect:TAT}

In this section, we consider the so-called circular mean Radon
transform in 2D, which is a standard model for sectional imaging
setups of photoacoustic tomography with constant sound speed. We refer
to \cite{Elbau:2012df,
HaltmeierSchusterScherzer2005,KuchmentKunyanski2010} for overviews of
the mathematics behind PAT and TAT.  The forward transform is defined
by \bel{def:Mc} \Mc f(\xi,r) = \frac{1}{2\pi}\int_{u\in S^1}
f(\xi+ru)\d u,\quad (\xi,r)\in S\times (0,\infty),
\end{equation}
where $S$ is a smooth curve in $\R^2$. 

We will consider only functions and distributions on $\rtwo$ that are
supported inside the open unit disk $D=\sparen{x\in \rtwo\st \norm{x}<
1}$.  We will assume the detectors are on the circle $S=S^1$, and use
the parameterizations for vectors in $S^1$ and for circles
respectively
\[\th(\phi)=(\cos\phi,\sin\phi) \text{ for }\phi\in [0,2\pi],\quad
C(\phi,r) = \sparen{x\in \rtwo\st \norm{x-\th(\phi)}=r}.\] The
tomographic data will have the following parametrization
\bel{def:Mc:circle}\text{for }\ (\phi,r)\in
\Xi: = [0,2\pi]\times (0,\infty), \quad g(\phi,r)=\Mc
f(\phi,r):=\frac{1}{2\pi r}\int_{x\in C(\phi,r)} f(x)\d x,\ee where we identify
$0$ and $2\pi$ or, equivalently, consider only functions and
distributions on $\Xi$ that are $2\pi$ periodic in $\phi$.  Note that
the measure $\d x$ in this integral is the arc length measure. 

Then, the dual transform to $\Mc$ (using the standard measures on $D$
and $\Xi$) is \[\Mst g(x) = \int_{0}^{2\pi}
g(\phi,\norm{x-\th(\phi)})\frac{1}{2\pi \norm{x-\th(\phi)}}\d \phi.\]

Because $D$ is the open disk with boundary $S^1$, $\Mc:\Dc(D)\to
\Dc(\Xi)$ is continuous.  Therefore, its adjoint, $\Mst:\Dc'(\Xi)\to
\Dc'(D)$ is weakly continuous.  Similarly, $\Mc:\Ec'(D)\to \Ec'(\Xi)$
is weakly continuous.

We will consider the limited data problem for this transform
specifying circular means with centers $\th(\phi)$ for $\phi \in
[a,b]$ with $b-a<2\pi$.  We define
\begin{equation}
\label{eq:restricted spherical mean transform} 
  \Mc_{[a,b]} f =
\chi_{A}\cdot\Mc f\ \ \text{ where $A=[a,b]\times (0,\infty)$}\,.
\end{equation}
  The wavefront
of $\chi_{A}$ is given by
\begin{equation}
\label{eq:WF hard cutoff} \WF(\chi_A) =
\sparen{((\phi,r),\nu\dphi)\st
\phi\in\sparen{a,b},\nu\neq0,r>0}.
\end{equation}

Before we state the main theorem of the section, we need to define
several concepts related to the microlocal analysis of the circular
mean transform.  For $\phi\in[0,2\pi]$ and $x\in D$, let
\bel{def:n}\npx = \frac{x-\th(\phi)}{\norm{x-\th(\phi)}}\,.\ee Then
$\npx$ the outward unit normal vector at $x$ to the circle
$C(\phi,\norm{x-\th(\phi)})$ centered at $\th(\phi)$ and containing
$x$.

For $A\subset [0,2\pi]$ we define the set \bel{def:Va} \Vc_A :=
\sparen{(x,\xi\dx)\in T^*(D)\st \exists \phi\in A,\exists \alpha \neq
0, \xi = \alpha\npx}.\ee In our next theorem, we show $\Vab$
is the set of \emph{possible visible singularities}, i.e., those that
can be imaged by $\Mst P \Mab $ (singularities for $\phi\in
\sparen{a,b}$ might be cancelled).

For $f\in \Ec'(D)$, we define the set \bel{def:Aab}\begin{aligned}
\Aab(f):=&\big\{(x,\xi\dx)\in T^*(D)\st \exists \phi\in \sparen{a,b}\exists r\in
(0,2),\exists \alpha\neq 0, \\&\qquad
\exists \tx\in C(\phi,r)\cap D,\
\paren{\tx,\alpha\n(\phi,\tx)\dx}\in WF(f)\\
&\qquad \qquad\text{ and }x\in C(\phi,r)\cap D, \xi=\alpha\n(\phi,x)
\big\}.\end{aligned}\ee The set $\Aab(f)$ will include \emph{added
artifacts} in the reconstruction operator $\Mst P \Mab (f)$.

\begin{theorem}[Visible and Added Singularities]\label{thm:visible and
added sing}  Let $f\in \Ec'(D)$
and let $P$ be a pseudo\-differential operator on $\Dc'(\Xi)$.  Then, 
\bel{circular mean singularities}
 \WF(\Mst P \Mab f)\subset \paren{\WF(f)\cap \Vab}\cup \Aab(f).\ee
\end{theorem}

One would expect that an inclusion $\WF(f)\cap
\Vc_{(a,b)}\subset\WF(\Mst P \Mab f)$ holds, which is true for sonar,
as proven in Section \ref{sect:sonar}, and for limited angle x-ray
tomography with reconstruction operators considered in
\cite{FrikelQuinto2013}.  We will discuss why this is not possible, in
general, for the PAT transform in Remark \ref{remark:elliptic}.

\begin{remark}\label{remark:A and V}  In general, Radon transforms
detect singularities conormal to the set being integrated over, and so
the visible singularities should be conormal to circles in the data
set.  The set $\Vab$ is the collection of conormals to circles
$C(\phi,r)$ for $\phi\in [a,b]$, and, according to Theorem
\ref{thm:visible and added sing}, $\Vab\cap \WF(f)$ is the set of
possible visible singularities in $\Mst P \Mab f$.  This idea of
visible singularities being (co)normal to the manifold of integration
is well-known and is for example, discussed in \cite{Palamodov:JFAA,Quinto93}.  This follows from the pioneering work of Guillemin
\cite{Gu1975, GS1977} showing Radon transforms are FIO associated to a
conormal bundle.

There are also possible added singularities, $\Aab(f)$, and they come
about from the data at the limits of the angular range, $a$ and $b$.
Remarkably, if $f$ has a singularity at one point on the circle
$C(\phi,r)$ (for $\phi=a$ or $\phi =b$) conormal to the circle, then
that singularity can be spread over the entire circle in the
reconstruction.  

Figure \ref{fig:illustration of artifacts} illustrates this perfectly
(as do the reconstructions in Section \ref{sect:numerics}).  The data
are taken for $\phi\in [0,\pi]$.  Every singularity in $D$ is visible
since, for every $x\in D$ and $\xi\in \rtwo\smo$, there is a $\phi\in
[0,\pi]$ with $\xi=\alpha \n(\phi,x)$ for some $\alpha \neq 0$; that
is, for every covector $(x,\xi\dx)\in T^*(D)\smo$, there is a $\phi\in
[0,\pi]$ such that $(x,\xi\dx)$ is conormal to the circle
$C(\phi,\norm{x-\th(\phi)})$.

The added singularities are as predicted by the theorem. They appear
on circles $C(\phi,r)$ when $\phi = 0,\pi$, the endpoints of the
interval of detectors and when some singularity of $f$ is conormal to
the circle.  For the simple phantom in Figure \ref{fig:illustration of
artifacts} there are four circles of added artifacts, and they are
tangent to the boundary of the object at points on the horizontal
axis.
\end{remark}

Next, we will show that the same artifact reduction procedure suggested in
\cite{FrikelQuinto2013} is valid for this case, too.  Let $a'$ and
$b'$ be chosen so $a<a'<b'<b$ and choose a smooth cutoff function
$\vp:[0,2\pi]\to\rr$ supported in $(a,b)$ and equal to one on
$[a',b']$.  Define the operator $\Kvp:\Dc'(\Xi)\to \Dc'(\Xi)$ by
\bel{def:K}\Kvp g = \vp g.\ee Then $\Kvp \Mc$ uses only data for
$\phi\in [a,b]$ but it provides a smooth cutoff.  This discussion
leads to the following theorem.

\begin{theorem}[Reduction of Artifacts for the circular mean
transform]\label{thm:reduction:circle} Let $\Kvp$ be defined by
\eqref{def:K} and let $P$ be a pseudodifferential operator on
$\Ec'(\Xi)$.  Finally, let \[\Lvp = \Mst P \Kvp \Mc, \] then $\Lvp$ is
a standard smooth pseudodifferential operator.  So, if $f\in \Ec'(D)$,
then \bel{reduced WF} \WF(\Lvp(f))\subset \WF(f)\cap
\Vab.\ee\end{theorem} 

Thus, only visible singularities of $f$ are visible in $\Lvp(f)$, and
there are no added singularities.

The proofs of these theorems are in the appendix.

\begin{figure}
	\centering
	\includegraphics[width=4cm]{circ.png}\hskip1ex %\subfloat[Original]{}
	\includegraphics[width=4cm]{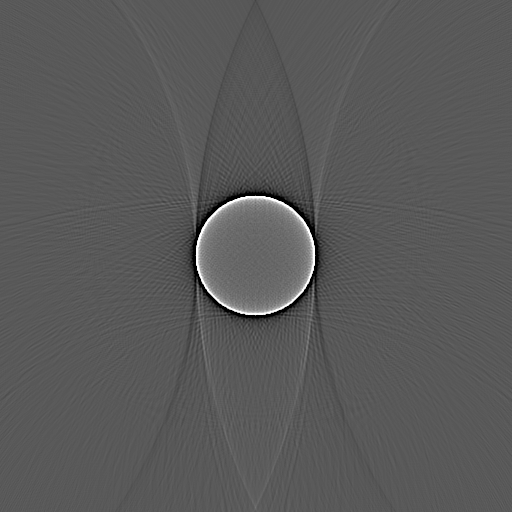}%\hskip1ex \subfloat[$\Lc^2 g$]{}
%	\subfloat[$\Ac_{\sparen{0,\pi}}$]
	\raisebox{-0.4pt}{
%	{
	\begin{tikzpicture}[scale=0.6667]
		% Lim angle sphere
		\draw [thick](-5cm,0) circle (0.75cm);
			
		% Rectangle
		\draw [thin, gray] (-8,-3) -- (-2,-3) -- (-2,3) -- (-8,3) -- (-8,-3);
			
		% Artifacts		
		\draw [dashed] (-7,3) arc (45:-45:4.25cm);
		\draw [dashed] (-5.1,3) arc (31.5:-31.5:5.75cm);
		\draw [dashed] (-3,3) arc (135:225:4.25cm);
		\draw [dashed] (-4.9,3) arc (148.5:211.5:5.75cm);
	\end{tikzpicture}
%	}
	}
	\caption{Lambda type reconstruction (middle) of the characteristic function of a circle (left) for the angular range $[0,\pi]$ and visualization of the set of added artifacts $\Ac_{\sparen{0,\pi}}$ (right). The correspondence between practical reconstruction and theoretical description \eqref{def:Aab} is remarkable.}
	\label{fig:illustration of artifacts}
\end{figure}

 \subsection{Sonar (spherical mean transform with centers on a
 plane)}\label{sect:sonar}

In this section, we analyze a limited data problem for sonar.  We
assume the sound speed is constant and there are not multiple
reflections.  Then, using the ansatz of Cohen and Bleistein
\cite{CB1979} the data can be reduced to integrals over spheres
centered on the ocean surface of the perturbation, $f(x)$, from the
constant sound speed.  We assume the ocean surface is planar and
consider sonar data from transceivers on a compact subset, $K$, of
that plane.

The functions we consider will be compactly supported in the open
upper half plane \[X:=\sparen{(x_1,x_2,x_3)\in \rthree\st x_3>0}.\]
Let $\Xis:= \rtwo\times (0,\infty),$ and denote the sphere centered at
$(y,0)$ and of radius $r$ by \[S(y,r)=\sparen{x\in \rthree\st
\norm{x-(y,0)}=r}.\]  The spherical mean transform of $f\in \Dc(X)$ is
denoted \bel{def:Ms}\Ms f(y,r) = \frac{1}{4\pi r^2}\int_{x\in
S(y,r)}f(x)\d x,\ee where $\d x$ is the surface measure on the sphere
$S(y,r)$.  For $g\in \Ec'(\Xis)$, the dual transform is given by 
\[\Msst g(x) = \int_{y\in \rtwo} g(y,\norm{x-(y,0)}) \frac{1}{4\pi
\norm{x-(y,0)}^2}\d y,\] where we note that the weight does not blow
up since the integral is evaluated at $x\in X$.

Let $K$ be a compact subset of $\rtwo$ with nontrivial interior.  We
consider the limited data problem where data are given over spheres
$S(y,r)$ with $(y,r)$ in \[A=K\times (0,\infty).\] The resulting
transform is \[\Msa = \chi_A \Ms.\] Since $K$ is compact, for any
$f\in \Ec'(X)$, $\Msa(f)\in \Ec'(\Xis)$.

For $x\in X$ and $y\in \rtwo$, we let \bel{def:nyx}\nyx =
\frac{x-(y,0)}{\norm{x-(y,0)}}\ee denote the outward unit normal at
$x$ to the sphere $S(y,\norm{x-(y,0)})$.  Note that the third
coordinate of $\nyx$ is never equal to zero since $x_3>0$. 

We define the set \bel{def:Vsk} \Vsk := \sparen{(x,\xi\dx)\in
T^*(X)\st \exists y\in K,\exists \alpha \neq 0, \xi = \alpha\nyx}.\ee
In our next theorem, we show $\Vsk$ is the set of \emph{possible
visible singularities}, i.e., those that can be imaged by $\Msst P
\Msa$ (singularities above points $y\in \bd(K)$ might be cancelled).

If $x=(x_1,x_2,x_3)\in \rthree$ then we define $x'=(x_1,x_2)$. 

Singularities are spread in more subtle ways in sonar than in TAT and
so we need to introduce more notation to properly describe these added
artifacts.  Let $S^2_+$ denote the open top hemisphere of $S^2$.  Let
$(y,r)\in \Xis$ and let $\eta\in \rtwo\smo$ and $\n\in S^2_+$.  Let
$C(y,r,\eta,\n)$ be the circle on $S(y,r)$ \bel{def:C}C(y,r,\eta,\n )
= \sparen{x\in S(y,r)\st \exists a\in \rr,\ x'-y=a\eta +r{\n'}}.\ee
This circle is the intersection of $S(y,r)$ with the vertical plane
that is parallel to $\eta$ and goes through $\paren{y+r\n',0}$. 

Let $K\subset\rtwo$ be a compact set bounded by a piecewise
$C^\infty$, simple, closed curve, $B=\bd(K)$.  Then, the singular
support of $\chi_K$ is $B$.  As noted in Example \ref{ex:char}, at
points $y\in B$ at which $B$ is a smooth curve, $(y,\eta\dy)\in
\WF(\chi_K)$ if and only if $\eta$ is normal to $B$ at $y$.  On the
other hand, if $B$ has a corner at $y$ then all covectors above $y$
are in $\WF(\chi_K)$. By \cite{Petersen}, $\WF(\chi_K)$ is a set
of covectors above $\bd(K)$, and at any point $y\in \bd(K)$ at which
$\bd(K)$ is smooth, they are the conormal covectors to $\bd(K)$ at
$y$ (see the discussion in Example \ref{ex:char}).  

The following definition allows us to apply our next theorem to more
general sets $K$ than those with piecewise smooth boundaries.

\begin{definition}[Generalized Normal Bundle] \label{def:generalized
normal bundle} Let $K$ be a compact subset of $\rtwo$.  Define
\emph{the generalized normal bundle of $K$ to be the set, $N(K)$, of
vectors in $\rtwo\times \paren{\rtwo\smo}$ corresponding to  covectors in the
wavefront set of $\chi_K$: }
\[N(K) = \sparen{(y,\eta)\in \rtwo\times \paren{\rtwo\smo}\st (y,\eta\dy)\in \WF(\chi_K)}.\]
\end{definition}

Using this notation and with $A=K\times (0,\infty)$, the wavefront set
of $\chia$ as a function on $\Xi_\Sc$ is \[\WF(\chia) =
\sparen{((y,r),\eta\dy)\st (y,\eta)\in N(K), r>0}.\] 

If $K$ is bounded by a smooth closed curve then our definition of
$N(K)$ corresponds with the standard definition of the normal bundle
of $\bd(K)$. However, if $K$ itself is a curve, then $\chi_K=0$ as a
distribution.  This means that, by our definition $N(K)=\emptyset$,
and this definition does not exactly correspond to the standard normal
bundle of $\bd(K)=K$.

Now, we introduce the set of added singularities. Let $f\in \Ec'(X)$
and let $K$ a compact subset of $\rtwo$.  Define the set
\bel{def:Ask}\begin{aligned} \Ask(f):=&\bigcup
\Big\{\big\{(x,\alpha\n(y,x)\dx) \st x\in
C(y,r,\eta,\n(y,\tx))\big\}\\&\qquad \st (y,\eta)\in N(K),\ r>0,\
\alpha\neq 0,\ \tx\in S(y,r)\ \text{ and
}\paren{\tx,\alpha\n(y,\tx)\dx}\in WF(f)\Big\}.\end{aligned}\ee

We will say more about this set
after the theorem. 

\begin{theorem}[Visible and Added Singularities for the Spherical
transform]\label{thm:visible and added sing sphere} Let $f\in \Ec'(X)$
and let $P$ be a properly supported pseudo\-differential operator on
$\Dc'(\Xi_{\Sc})$.  Let $K$ be a compact subset of $\rtwo$ with
nontrivial interior and let and $A=K\times (0,\infty)$.  Then,
\bel{spherical mean singularities} \WF(\Msst P \Msa f)\subset
\paren{\WF(f)\cap \Vs{K}}\cup \As{K}(f).\ee If $P$ is elliptic on
$\range(\Ms)$, then,
\bel{elliptic spherical}
\WF(f)\cap\Vs{\intt(K)}\subset
\WF(\Msst P \Msa f)\,.\ee
\end{theorem}

\begin{remark}\label{remark:wild singularities}
In practice, $K$ will be a compact set bounded by a simple piecewise
smooth curve, and we now consider $\Ask(f)$ in this case.  By the
definition of $\Ask(f)$ and Theorem \ref{thm:visible and added sing
sphere}, singularities are added on spheres $S(y,r)$.  For added
singularities to appear on $S(y,r)$ the following must be satisfied:
\begin{itemize}
\item $y\in \bd(K)$

\item There is an $\tx\in S(y,r)$ and $\alpha \neq 0$ with
$(\tx,\alpha\n(y,\tx))\in \WF(f)$.

\end{itemize}
Once this is true, the singularities spread differently depending on
the geometry of $\bd(K)$.  Let $y\in \bd(K)$, $r>0$, and $\alpha\neq
0$.  Assume $\tx\in S(y,r)$ with $(\tx,\alpha\n(y,\tx)\dx)\in \WF(f)$.

First, assume $\bd(K)$ is a smooth curve at $y$.  Let $\eta$ be a
normal to $\bd(K)$ at $y$, then all normals to $\bd(K)$ at $y$ are
parallel to $\eta$.  Thus, the added singularities caused by the
singularity of $f$ at $(\tx,\alpha\n(y,\tx)\dx)$ will be on the
semicircle $C(y,r,\eta,\alpha\n(y,\tx))\cap X$ but nowhere else on
$S(y,r)$.

Now assume that $\bd(K)$ has a corner at $y$.  Then, for all $\eta\neq
0$,$(y,\eta)\in N(K)$.  Each such $\eta$ generates a semicircle of
possible added singularities that is in a plane parallel to $\eta$ and
through the fixed point $(y+r\n'(y,\tx),0)$.  As $\eta$ changes, this
semicircle sweeps out the entire hemisphere $S(y,r)\cap X$. Thus, in
this case, added singularities are on this entire hemisphere.

This discussion justifies using measurement sets $K$ with smooth
boundaries so added singularities do not spread along entire
hemispheres.
\end{remark}

These theorems do not address what happens to ``boundary
singularities'' of $f$, namely those for covectors $(x,\xi\dx)\in
\Vc_{\bd(K)}$.  In general, these singularities can be invisible or
visible.

The reconstruction from simulated sonar data in Figure
\ref{fig:QRS2011}, which was taken from \cite{QRS2011}, illus\-trate
our theorem perfectly; there are added singularities in exactly the places
predicted by the theorem.  In fact, all of the reconstructions in
\cite{QRS2011} have the added artifacts in the locations predicted by
the theorem.
\begin{figure}[ht]
\centerline{
\includegraphics[width=5cm]{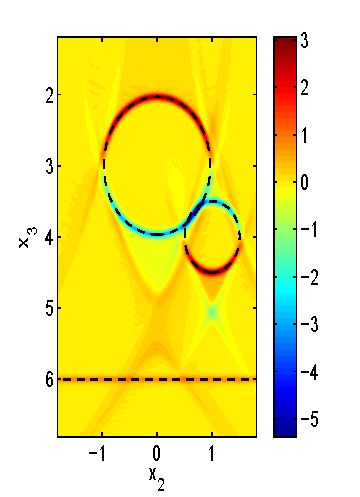}}
\caption{Reconstruction from spherical mean data of the characteristic
functions of two balls and a horizontal ocean floor.  The center set
$K=[-12,12]^2$ and the reconstruction is in the plane $x_1=0.25$. The
spheres are centered at $(0,0,3)$ and $(0,1,4)$. The added artifacts
in this picture are along circles parallel the $x_2\,x_3$ plane (since
they are caused by normals perpendicular to the $x_1$ axis for spheres
centered on lines parallel the $x_1$ axis).  Note that the predicted
artifact caused by the ocean floor would be outside the reconstruction
region since it would be caused by spheres centered at $(0.25,\pm
12,0)$ and of radius 6.  Reprinted from \cite{QRS2011} by permission
of IOP Publishing.} \label{fig:QRS2011}
\end{figure}

The same artifact reduction procedure suggested in
\cite{FrikelQuinto2013} is valid for this case, too.  Let $K'$ be an
open set whose closure is contained in $\intt(K)$ and let
$\vp:\rtwo\to\rr$ be a smooth cutoff function that is supported in
$\intt(K)$ and equal to one on $K'$. In this case, the operator
$\Kvp:\Dc'(\Xi_{\Sc})\to \Dc'(\Xi_{\Sc})$ is
\[\Kvp g = \vp g.\] Then $\Kvp \Ms$ uses only data for
$y\in K$ but it provides a smooth cutoff.  This discussion leads to
the following theorem.

\begin{theorem}[Reduction of Artifacts for the sonar
transform]\label{thm:reduction:sphere} Let $K$ be a compact subset of
$\rtwo$ with nontrivial interior.  Let $P$ be a properly supported
pseudodifferential operator on $\Ec'(\Xi_{\Sc})$.  Let $K'$ be an open
set whose closure is contained in $\intt(K)$ and let $\vp:\rtwo\to\rr$
be a smooth function supported in $\intt(K)$ and equal to one on $K'$.
Let \[\Lvp = \Msst P \Kvp \Ms, \] then $\Lvp$ is a standard $C^\infty$
pseudodifferential operator and so 
\[\WF(\Lvp(f))\subset \WF(f)\cap \Vsk\] for $f\in \Ec'(X)$.  If $P$ is
elliptic on $\range(\Ms)$, then $\Lvp$ is
elliptic on $\Vs{K'}$, so, if $f\in \Ec(D)$, then \bel{reduced WF}
\WF(f)\cap \Vs{K'}\subset\WF(\Lvp(f))\subset \WF(f)\cap \Vsk.
\ee\end{theorem} 

Thus, only visible singularities of $f$ are visible in $\Lvp(f)$, and
there are no added singularities.  In contrast to the circular mean
case, when $P$ is elliptic, all of the singularities of $f$ in
$\Vs{K'}$ are visible in $\Lvp( f)$.  This difference will be
explained in Remark \ref{remark:elliptic}.

\section{Numerical Examples}\label{sect:numerics}

In this section we illustrate the capability of our artifact reduction
strategy for limited view reconstruction from circular mean data. In
particular, we show that the proposed artifact reduction strategy
performs very well on synthetic data as well as on real photoacoustic
data. In our experiments we consider the circular mean transform, $\Mc$, with detectors on a circle or
circular arc surrounding the object.  This is the standard model in
sectional imaging in photo- and thermoacoustic tomography
\cite{Elbau:2012df,razansky2009multispectral}, and it is the model we analyzed in Section
\ref{sect:TAT}.
The Sonar case, as presented in Section \ref{sect:sonar}, can be implemented similarly
and we expect similar results. 

For the numerical experiments, we implemented the following
reconstruction operators in Matlab,
\begin{equation}
\label{eq:rec op aux}
	\Lc^1 = \Mst \paren{\frac{\mathrm{d}}{\mathrm{d}r}}\Mc,\quad \Lc^2 = \Mst \paren{-\frac{\mathrm{d}^2}{\mathrm{d}r^2}}\Mc.
\end{equation}
We also implemented their artifact reduced versions 
\begin{equation}
\label{eq:rec op mod aux}
	\Lvp^1 = \Mst \paren{\frac{\mathrm{d}}{\mathrm{d}r}}\Kvp\Mc,\quad \Lvp^2 = \Mst \paren{-\frac{\mathrm{d}^2}{\mathrm{d}r^2}}\Kvp\Mc,
%	\Lvp = \Mst P \Kvp \Mc
\end{equation}
where $\Kvp$ is the multiplication operator defined in \eqref{def:K}.
$\Kvp$ multiplies the limited view data $g=\Mc_{[a,b]} f$ with a
smooth truncation function $\varphi$ that satisfies the assumptions of
Theorem \ref{thm:reduction:circle}. In our implementation of $\Kvp$ we
use a similar cutoff function as in \cite{FrikelQuinto2013}. Namely,
we choose $\eps\in \paren{0,(b-a)/2}$ let
$\vp=\vp_\epsilon:[0,2\pi]\to\rr$ be supported in $[a,b]\subset [0,
2\pi]$ so that $\vp(\phi)=1$ for $\phi\in[a+\eps,b-\eps]\subset[a,b]$.
In the transition regions, $[a,a+\eps]$ and $[b-\eps,b]$, we used the
function $\nu_\eps$ defined by
$\nu_\eps(x)=\exp(\frac{x^2}{x^2-\epsilon^2})$ for $\abs{x}<\eps$ and
$\nu_\eps(x)=0$ for $\abs{x}\geq \eps$, to generate a smooth
transition from 0 to 1 and from 1 to 0, respectively. That is, for
$\phi\in[a,a+\eps]$ we set $\vp_\eps(\phi)=\nu_\eps(a+\eps-\phi)$, and
for $\phi\in[b-\eps,b]$ we set $\vp_\eps(\phi)=\nu_\eps(\phi-b+\eps)$.
The function $\vp_\eps$ defines a smooth function apart from
$\phi=a+\eps$ and $\phi=b-\eps$. Even though this function is not
globally smooth, we may use it for artifact reduction because, in
practice, it is evaluated at a finite number of points and there is a
smooth function that has these values at these points. Moreover, this
function has shown to provide good artifact reduction results in
limited angle x-ray tomography, cf. \cite{FrikelQuinto2013}. 
	
In our experiments, we performed two sets of reconstructions.  In our first experiment, we computed reconstructions from synthetic data. Here, we generated limited view spherical mean data $g=\Mc_{[a,b]} f$ in Matlab  from a characteristic function of a circle centered around the origin and from the Forbild head phantom \cite{Forbild_head_phantom}. The corresponding Lambda type reconstructions $\Lc^2 g$ and $\Lvp^2 g$ are shown in Figures \ref{fig:pat circ rec} and \ref{fig:pat phantom rec}. The standard reconstructions $\Lc^2 g$ without artifact reduction (left images) clearly show the circular artifacts as characterized in Section \ref{sect:TAT}, see also Figure \ref{fig:illustration of artifacts}. In particular, in the $\Lc^2 g$ reconstruction of Figure \ref{fig:pat phantom rec}, we observe that many artifacts overlap and significantly degrade the reconstruction quality (even generating new image features). This is due to the presence of many singularities in the original image. Thus, similar behavior can be expected for any limited view reconstruction of an image with many singularities. By using artifact reduced reconstruction operators $\Lvp^2 g$, in Figures \ref{fig:pat circ rec} and \ref{fig:pat phantom rec}, we can clearly observe an improvement of image quality. The artifacts are clearly reduced while most of the visible singularities are preserved (even this could not be proven in Section \ref{sect:TAT}, see Remark \ref{remark:elliptic}). However, we also observe that some of the visible singularities with directions at the boundary of the limited view are smoothed.

\begin{figure}[t]
\centering
	\includegraphics[width=4cm]{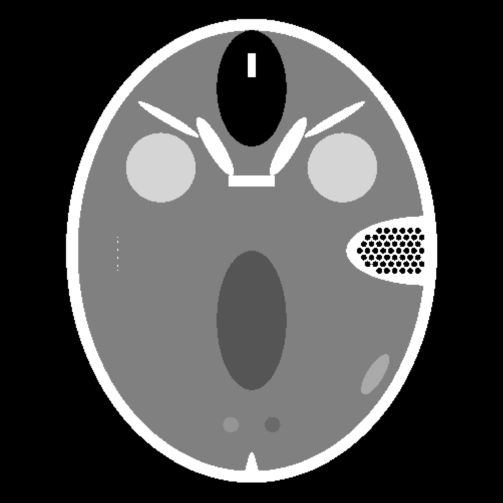}\hskip2ex %\subfloat[Original]{}
	\includegraphics[width=4cm]{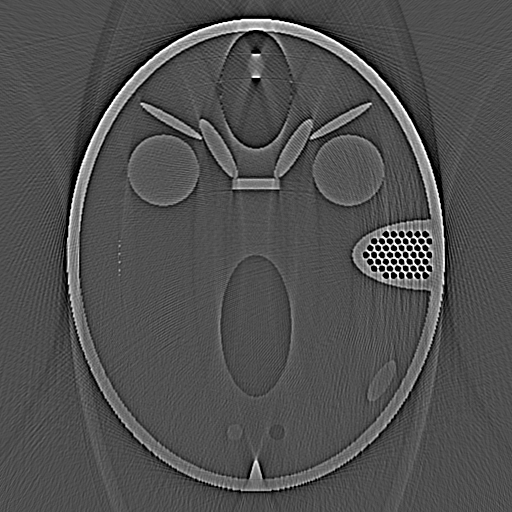}\hskip2ex %\subfloat[$\Lc^2 g$ - no artifact reduction]{}
	\includegraphics[width=4cm]{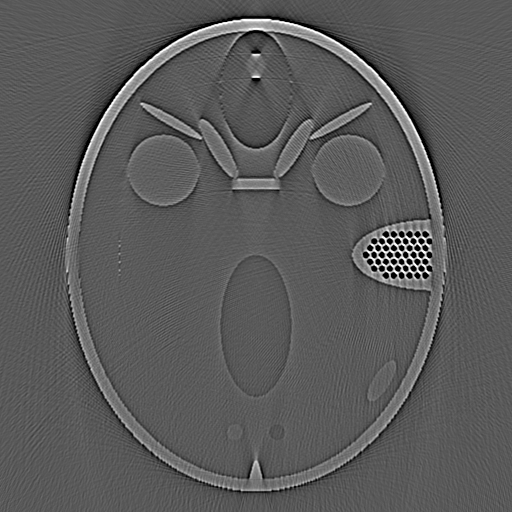} %\subfloat[$\Lvp^2 g$ - with artifact reduction]{}
	\caption{Lambda type reconstruction ($512\times 512$) of the FORBILD head phantom \cite{Forbild_head_phantom} for the limited angular range $[0^\circ,180^\circ]$ ($750$ projections, $725$ radii). The left-hand picture is the original phantom, the middle picture is the
reconstruction without artifact reduction and the right-hand picture is an artifact reduced reconstruction with $\eps=18^\circ$. % with $\eps = ??$ %$\eps=0.1\!\cdot\!\pi$ \emph{Left:} Without artifact reduction. \emph{Middle:} With artifact reduction. \emph{Right:} Difference image. 
%In above reconstructions, one can clearly observe the effect of artifact reduction.
}
\label{fig:pat phantom rec}
\end{figure}

\begin{figure}[t]
\centering
	\includegraphics[width=4cm,
	angle=180]{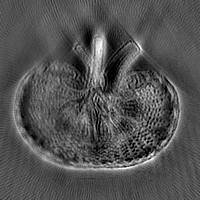}\hskip2ex  %\subfloat[$\Lc^1 g$ - no artifact reduction]{}
	\includegraphics[width=4cm,
	angle=180]{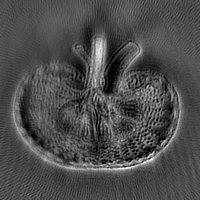}\hskip2ex
	 %\subfloat[$\Lvp^1 g$ - with artifact reduction]{}
	\includegraphics[width=4cm,
	angle=180]{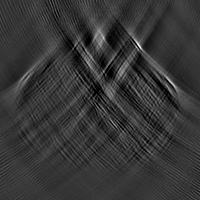}  %\subfloat[Difference image $\abs{\Lc^1 g-\Lvp^1 g}$]{}
	\caption{Real data reconstruction ($200\times200$) of a paper
	phantom with ink as an optical absorber for the limited
	angular range $[-45^\circ,225^\circ]$ ($256$ projections,
	$2030$ radii).  
  %by the reconstruction operators $\Lc^1$ and $\Lvp^1$.  
	\emph{Left:} $\Lc^1 g$ (no artifact reduction); \emph{Middle:} $\Lvp^1 g$ with $\eps=45^\circ$ (reconstruction with artifact reduction); \emph{Right:} $|\Lc^1 g-\Lvp^1 g|$ (difference image). %$\eps=\pi/4$
%	\emph{Left:} Without artifact reduction. \emph{Middle:} With artifact reduction. \emph{Right:} Difference image. 
In above reconstructions one can clearly observe the effect of artifact reduction. The difference image shows that the geometry of artifacts corresponds to the theoretical characterization of Section \ref{sect:TAT}.}
\label{fig:pat real data rec1}
\end{figure}
\begin{figure}[t]
\centering
	\includegraphics[width=4cm]{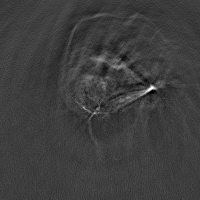}\hskip2ex 
%\subfloat[$\Lc^1 g$ - no artifact reduction]{}
	\includegraphics[width=4cm]{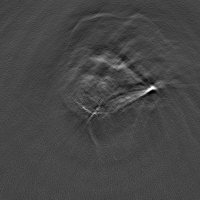}\hskip2ex 
%\subfloat[$\Lvp^1 g$ - with artifact reduction]{}
	\includegraphics[width=4cm]{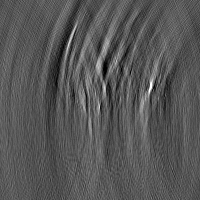}
%\subfloat[Difference image $\abs{\Lc^1 g-\Lvp^1 g}$]{}
\caption{Real data reconstruction ($200\times200$, $2193$ radii) of a
mouse tumor for the limited angular range $[0^\circ,180^\circ]$ (200 projections).  % by the reconstruction operators $\Lc^1$ and $\Lvp^1$. 
\emph{Left:} $\Lc^1 g$ (no artifact reduction); \emph{Middle:} $\Lvp^1 g$ with $\eps=45^\circ$ (reconstruction with artifact reduction); \emph{Right:} $|\Lc^1 g-\Lvp^1 g|$ (difference image). Though in the left image the added artifacts are not clearly distinguishable from reliable image features one clearly observes an increase of image quality by using the reconstruction operator with artifact reduction. The difference image shows that artifacts are effectively removed and that the geometry of the artifacts is in accordance with the theoretical characterizations, see also Figure \ref{fig:illustration of artifacts}.}
\label{fig:pat real data rec2}
\end{figure}

In our second experiment, we computed reconstructions from real
photoacoustic measurements. More precisely, we computed the
backprojection of the experimental pressure data that were generated
through a focused illumination of the object in the plane $z=0$, and
measured by acoustic detectors that were distributed on a circular arc
surrounding the object. The data is by courtesy of the group of Prof.
Daniel Razansky (Institute for Biological and Medical Imaging,
Helmholtz Zentrum M\"unchen). For more details on the measurement
setup we refer to \cite{razansky2009multispectral}. Although real PAT
data is three dimensional in nature, the pressure data $p$ (measured
in the plane $z=0$) of the described sectional imaging setup is
related to the circular mean operator by
$p(\xi,t)=\frac{1}{2}\partial_t\Mc f(\xi,t)$, where $\Mc f(\xi,t)$
denote the  circular means of the imaged section, for more details
see \cite{Elbau:2012df}. Thus, our reconstructions correspond to the
application of the operators $\Lc^1$ and $\Lvp^1$ to the limited view
circular mean data $g=\Mc_{[a,b]} f$. 

The images of Figure \ref{fig:pat real data rec1} show reconstructions of a paper phantom
(which has ink as an optical absorber) and images in Figure
\ref{fig:pat real data rec2} show reconstructions of a mouse tumor. In
Figure \ref{fig:pat real data rec1}, one can clearly observe the
effect of artifact reduction induced by smooth truncation of the
limited view data at the ends of the angular range. As previously, the visible
singularities are well preserved while the circular artifacts are
removed as can be seen from the difference image in Figure
\ref{fig:pat real data rec1}. We would like to point our that even
though the ground truth is not available the theoretical analysis of
Section \ref{sect:characterization} together with previous examples
enable us to differentiate between reliable image features and added
artifacts (which are located on circular arcs). This is very important
since there are practical situations where it is not that easy to
distinguish between reliable image features and artifacts. For
example, Figure \ref{fig:pat real data rec2} shows another limited
view reconstruction from experimental photoacoustic data where the
artifacts are not as distinctive as in Figure \ref{fig:pat real data
rec1}. The reconstruction with artifact reduction is presented in the
middle image of Figure \ref{fig:pat real data rec2}. Here, one can
observe that a superior image quality is achieved through the application of artifact reduction: the reconstructions
appear to be more clear, having a more homogeneous background. As seen from the difference image in Figure
\ref{fig:pat real data rec2} the circular artifacts are removed
effectively.

Finally, we note that the visibility of artifacts in numerical
reconstructions highly depends on the choice of the $\eps$-parameter
in the cutoff function ($\vp_\eps$ above). For a detailed discussion
we refer to \cite{FrikelQuinto2013}. Here, we note only that there is
a trade-off between mitigation of artifacts and smoothing of visible
singularities: large $\eps$-parameters lead to good mitigation of
artifacts in the reconstruction while at the same time (depending on
the limited view) a large $\eps$-parameter can smooth the visible
singularities, as in Figures \ref{fig:pat circ rec} and \ref{fig:pat
phantom rec}. In contrast, small $\eps$-parameters preserve most of
the visible singularities while circular artifacts might still be
clearly visible. In all of the above experiments, we have chosen the
\enquote{optimal} $\eps$-parameter by visually inspecting a series of
reconstructions with different parameters.

To sum up, our theoretical analysis of Section \ref{sect:characterization} helps to distinguish between reliable image features and artifacts in limited view reconstructions (thus improving proper interpretation of limited view reconstructions). In particular, it provides an effective and easy-to-implement strategy for artifact reduction. Our experiments show that this strategy significantly improves the reconstruction quality not only in reconstructions from synthetic phantom data but also from real experimental data.

\section{Concluding Remarks}\label{sect:conclusions}

In this paper we presented a paradigm to precisely characterize
potential added singularities in limited data tomography problems.
The paradigm in Section \ref{sect:general strategy} is general and can
be applied to a range of other limited data problems besides limited
angle tomography, and the ones we studied in this article: PAT/TAT and
Sonar.  We proved that these added singularities come from data at the
boundary of the data set.  For PAT/TAT with data on $[a,b]\times
(0,2)$, if the object, $f$, has a singularity on a circle $C(a,r)$ or
$C(b,r)$ that is conormal to the circle, then the singularity can be
spread over the entire circle (see Remark \ref{remark:A and V}).  In
Sonar with limited data on $K\times (0,\infty)$, the added artifacts
appear on part or all of the spheres $S(y,r)$ for $y\in \bd(K)$.  If
$f$ has a singularity conormal to such a sphere, then singularities
can be spread along a circle on that sphere if $\bd(K)$ is smooth at
$y$ and along the entire sphere if $\bd(K)$ is not smooth at $y$ (see
Remark \ref{remark:wild singularities}).  

For the case of sonar we showed that detector locations (the set $A$)
with smooth boundaries will produce fewer added artifacts than sets
with corners. 

Moreover, we showed that, with a smooth cutoff at the boundary, the added
artifacts are eliminated.  Reconstructions from real and simulated
data were presented that illustrate our paradigm and the artifact
reduction procedure.

\section*{Acknowledgements}

The first author thanks Tufts University for its hospitality during
the during the spring semester 2014 when most of the presented work
was done. Moreover, he acknowledges support by the Helmholtz
Association within the young investigator group VH-NG-526.  The second
author was supported by NSF grant DMS1311558 as well as the generosity
of the Technische Universit\"at M\"unchen and the Helmholtz Zentrum,
M\"unchen. The second author thanks Jan Boman for many enlightening
discussions about microlocal analysis over the years as well a helpful
observation about wavefront of real-valued functions that is used in
the proof of Theorem \ref{thm:visible and added sing sphere}.
We thank Anuj Abhishek for his helpful comments on ideas in this
article. The authors thank Clifford Nolan for insightful discussions
about this research and seismic imaging.  They thank Frank Filbir and
Stefan Kunis for encouraging and supporting this collaboration.  The
authors thank Daniel Queiros and the group of Prof. Daniel Razansky
(both at the Institute for Biological and Medical Imaging, Helmholtz
Zentrum M\"unchen) for providing us experimental photoacoustic data.

Finally, the authors are indebted to the referees for their thoughtful
comments that improved the article.

\begin{appendix}

\section{Proofs of Theorems \ref{thm:visible and added
sing}, \ref{thm:reduction:circle}, \ref{thm:visible and added sing
sphere}, and  \ref{thm:reduction:sphere}}
\label{sect:appendix} 

 \begin{proof}[Proof of Theorem \ref{thm:visible and added sing}] We
prove the theorem by going through the paradigm discussed in section
\ref{sect:general strategy} using Lemma \ref{lemma:composition
identity}.  

 The canonical relation of $\Mc$ is given in equations (4.2)-(4.4) in
the proof of Lemma 4.3 on p.\ 396 of \cite{AQ1996a}.  \bel{canonical
relation circle}\begin{aligned} \Cm = \Big\{\big((\phi,r),&
\alpha\big[\th^\bot(\phi)\cdot\npx\dphi+\dr\big]; x,
\alpha\npx\dx\big)\st \\
	&\alpha\neq0, (\phi,r)\in \Xi, x\in
	C(\phi,r)\cap D\Big\}
\end{aligned}\ee where $\thperp(\phi) = \th(\phi+\pi/2)$ is
perpendicular to $\th(\phi)$.  

Recall that $A=[a,b]\times (0,\infty)$ in this proof.  The
non-cancellation condition \eqref{non-cancellation} in Theorem
\ref{thm:WF mult} holds since $\WF(\chi_A)$ has $\dr$ component of
zero, and any covector in $\Cm\circ \paren{T^*(D)\smo}$ has nonzero
$\dr$ component.  So, this theorem can be used to show that the
product $\chiab \Mc$ is well-defined for distributions $f\in \Ec'(D)$
and \[\WF(\chiab \Mc f) \subset \Qc(A,\Cm\circ\WF(f))\] (see
\eqref{composition-genl:Mf}).  We calculate
$\Qc(A,\WF(\Cm\circ\WF(f)))$ in steps.

Using \eqref{canonical relation circle} we see
\begin{multline}
	 \Cm\circ\WF(f) = \Big\{
\paren{(\phi,r),\alpha\big[\th^\bot(\phi)\cdot\npx\dphi+\dr\big]}\st
\\ (\phi,r)\in \Xi\,,\alpha\neq0,\ \exists x\in C(\phi,r),\
\paren{x,\alpha\npx\dx}\in\WF(f)\Big\}.
\end{multline}

According to the definition of $\Qc$, \eqref{def:Q}, \bel{eq:wavefront
set inclusion intersected sphere}
\begin{aligned} \Qc(A,\Cm\circ \WF(f))=& \big\{ ((\phi,r),
\xi+\eta)\st \phi\in [a,b],\ \bparen{((\phi,r), \xi)\in\Cm\circ
\WF(f)\textrm{ or }\xi=0}\,\\&\qquad \textrm{ and }\,\bparen{((\phi,r),
\eta)\in\WF(\chi_A)\textrm{ or }\eta=0}\big\}
\end{aligned}\ee 

 One can break the right-hand side of \eqref{eq:wavefront set
inclusion intersected sphere} into the union of three sets:
\begin{align}
\label{Qc(A,f)} \Qc(A,\Cm\circ \WF(f))&=
\bparen{\paren{\Cm\circ \WF(f)}\cap \sparen{(\phi,r,\eta)\in
T^*(\Xi)\st \phi\in [a,b]}}\\
&\quad \cup \WF(\chi_A) \cup \Wab(f) \end{align}
where 
the  set
\begin{align*} \Wab(f) = \Big\{ &\paren{(\phi,r), \bparen{\mu +
\alpha\th^\bot(\phi)\cdot\n(\phi,\tx)}\dphi + \alpha
\dr}
\st \\ &\quad\alpha,\mu\neq 0,  \phi\in\{a,b\},r\in (0,\infty)\ \exists \tx\in
C(\phi,r),\ 
\paren{\tx,\alpha\n(\phi,\tx)\dx}\in\WF(f)\Big\}.
\end{align*}  is generated by terms
 $\xi\neq 0$ and $\eta\neq 0$ in \eqref{eq:wavefront set inclusion
intersected sphere}.  Furthermore, $\Wab(f)$ can be written
\bel{def:Wab}\begin{aligned} \Wab(f) = \Big\{ \paren{(\phi,r),
\nu\dphi+\alpha\dr} \st& \phi\in\sparen{a,b},\ \nu\in \rr,
\alpha\neq0,\\&\exists \tx\in C(\phi,r),\
\paren{\tx,\alpha\n(\phi,\tx)\dx}\in\WF(f)\Big\}.
\end{aligned}\ee because
$\nu=\mu+\alpha\th^\bot(\phi)\cdot\n(\phi,\tx)$ is arbitrary since,
for $\phi\in \sparen{a,b}$, every nonzero covector
$((\phi,r),\mu\dphi)\in WF(\chi_A)$ (we allow $\mu=0$ in this
expression to make the proof easier to describe).  At this point, we
have finished steps \eqref{paradigm:start} to \eqref{paradigm:Mf} of the
paradigm of Section \ref{sect:general strategy}.

In the next step we compute $\Cmt \circ \Qc(A,\Cm\circ \WF(f))$ by
first finding explicit expressions for the compositions with $\Cm$ and
$\Cmt$. 

For $(x,\xi)\in D\times \paren{\rtwo\smo}$ define $\po$ as the unique
angle in $[0,2\pi]$ such that $\th(\po)$ is the intersection of $S^1$
with the ray $\sparen{x+t\xi\st t<0}$ and define $\pt$ as the unique
angle in $[0,2\pi]$ such that $\th(\pt)$ is the intersection of $S^1$
with the ray $\sparen{x+t\xi\st t>0}$.  Note that $\po$  and $\pt$ are
smooth functions of $(x,\xi\dx)$.   Define for $j=1,2$ 
\[\begin{gathered}\alpha_1(\xi) = \norm{\xi}\qquad 
\alpha_2(\xi) = -\norm{\xi},\qquad r_j(x,\xi) =
\norm{x-\th(\phi_j(x,\xi))},\\
c_j(x,\xi\dx) = \paren{(\phi_j(x,\xi),r_j(x,\xi)),
\alpha_j(x,\xi)\bparen{\n(\phi_j(x,\xi),x)\cdot
\thperp(\phi_j(x,\xi))\dphi + \dr}} \end{gathered}\] then
$\Cm\circ\sxxi = \sparen{c_j(x,\xi\dx)\st j=1,2}$.  

Let $\nu\in \rr,\alpha\neq 0$ and assume $\nu/\alpha\in [-1,1]$.
Define \bel{def:x}\begin{gathered}x((\phi,r),
\nu,\alpha)=\th(\phi)+r\paren{(\nu/\alpha) \thperp(\phi) -
\sqrt{1-(\nu/\alpha)^2}\th(\phi)}, \\
\xi((\phi,r),\nu,\alpha) =\alpha
\n(\phi,x((\phi,r),\nu,\alpha)).\end{gathered}\ee Note that
$x((\phi,r),\nu,\alpha)\in C(\phi,r)$ and if $\nu/\alpha$ and $r$ are
sufficiently close to zero, this point is in $D$.  Define
\[T_\Mc = \sparen{\paren{(\phi,r),\nu\d\phi+\alpha\dr}\st (\phi,r)\in
\Xi, \alpha\neq 0, \nu/\alpha \in
[-1,1],x((\phi,r),\nu,\alpha)\in D}.\] Then,
$c_j:T^*(D)\smo\to T_\Mc$ and $c_1(T^*(D)\smo)\cup
c_2(T^*(D)\smo)=T_\Mc$.  These statements are proven using geometry
and the observations that
$x_\phi=x((\phi,r),\nu,\alpha)$ satisfies
$(\n(\phi,x_\phi)\cdot\thperp(\phi) = \nu/\alpha$ and $x_\phi\in
C(\phi,r)$.

For $\nu/\alpha\in[-1,1]$ and $(\phi,r)\in \Xi$, whenever
$x((\phi,r),\nu,\alpha)\in D$,  define
\bel{cinv}\begin{gathered}\cinv((\phi,r),\nu,\alpha)=
\paren{x((\phi,r),\nu,\alpha),\xi((\phi,r),\nu,\alpha)\dx},\\
\text{then } \sparen{\cinv((\phi,r),\nu,\alpha)} =\Cmt\circ
\sparen{((\phi,r),\nu\dphi+\alpha\dr)}.\end{gathered}\ee These
calculations allow one to show for $j=1,2$ that $\cinv\circ c_j$ is
the identity map on $T^*(D)\smo$.  So, if $(x,\xi\dx)\in T^*(D)\smo$
then \bel{unique composition}\Cmt\circ
\Cm\sparen{(x,\xi\dx)}=\sparen{(x,\xi\dx)}.\ee

Now that these basic observations have been made, we can calculate the
wavefront set of $\Mst P\Mab(f)$.  Using the composition rules
\eqref{eq:distributive and associative relations for compositions},
Theorem \ref{thm:HS}, and \eqref{Qc(A,f)} yields the following union
of sets:
\begin{align}
\Cmt\circ\Qc(A,\Cm\circ WF(f))
&=\Cmt\circ\bparen{\paren{\Cm\circ\WF(f)} \cap
\sparen{(\phi,r,\eta)\in T^*(\Xi)\st \phi\in [a,b]}}\label{WF(f)}\\
&\qquad\qquad \cup \,\Cmt\circ \WF(\chi_A) \label{WFchi}\\ 
&\qquad \qquad\qquad\cup \Cmt\circ \Wab(f)\label{Wab}
\end{align}

We examine the three terms of \eqref{WF(f)}-\eqref{Wab} separately and
we first show that the set in equation \eqref{WF(f)} is equal to $\Vab
\cap \WF(f)$.  Let $(\xo,\xio\dx)$ be in the set in \eqref{WF(f)}.
Then $\xoxio\in\Cmt\circ\paren{\Cm\circ\WF(f)}= \WF(f)$ by
\eqref{unique composition}.  Because $\xoxio$ is also in the set \[
\Cmt\circ\sparen{(\phi,r,\eta)\in T^*(\Xi)\st \phi\in [a,b]},
\] either $\phi_1(\xo,\xio)$ or $\phi_2(\xo,\xio)$ 
(or both) must be in $[a,b]$.  Therefore, $\xio=\alpha\n(\phi,\xo)$
for some $\phi\in [a,b]$ and some $ \alpha\neq 0$.  This means that
$\xoxio\in \Vab$ and the set in \eqref{WF(f)} is contained in
$\Vab\cap\WF(f)$.  The reverse containment is proven in a similar way.

Next, we consider the set in \eqref{WFchi}.  As the $\dr$ component of
any covector in $\WF(\chi_A)$ is zero and covectors in $\Cmt$ all have
nonzero $\dr$ component, $\Cmt\circ\WF(\chi_A) = \emptyset$. 

Finally, we consider the set $\Cmt\circ \Wab(f)$. Let
$\paren{(\phi,r),\nu\dphi+\alpha\dr}\in \Wab(f)$.  Then $\phi\in
\sparen{a,b}$, $\nu\in \rr$ is arbitrary, $\alpha\neq 0$, and for some
$\tx\in C(\phi,r)$, $(\tx,\alpha\n(\phi,\tx)\dx)\in\WF(f)$.  For
$\Cmt\circ \paren{(\phi,r),\nu\dphi+\alpha\dr}$ to be nonempty,
$\nu/\alpha$ can be any value in $[-1,1]$ such that
$x_\phi=x((\phi,r),\nu,\alpha)\in D$.  Since $r$ is fixed, $x_\phi$ is
an arbitrary point on $C(\phi,r)\cap D$.  This means that for every
$x\in C(\phi,r)\cap D$, the covector $(x,\alpha\npx\dx)\in \Cmt\circ
\Wab(f)$. Therefore, the set of possible added singularities from
\eqref{Wab} is the set $\Aab(f)$ defined in \eqref{def:Aab}.

Finally, we use Lemma \ref{lemma:composition identity} and the  general
wavefront containment \eqref{composition-genl} to conclude
\eqref{circular mean singularities} in Theorem \ref{thm:visible and
added sing}.  This finishes the proof.
\end{proof}

\bigskip
\begin{proof}[Proof of  Theorem \ref{thm:reduction:circle}]
Because $\Kvp$ is a (trivial) pseudodifferential operator on
$\Dc'(\Xi)$, $\Kvp\Mc$ is a standard $C^\infty$ FIO with the same
canonical relation as $\Mc$.  Therefore, $\Lvp$ is a standard
pseudodifferential operator for distributions supported in $D$ because
\bel{Delta:circle}\Cmt\circ \Cm = \Delta\ee is the diagonal in
$T^*(D)$ (see \cite{Treves:1980vf}) by equation \eqref{unique
composition} .  So, $\WF(\Lvp f)\subset \WF(f)$ for any $f\in
\Ec'(D)$.

To show the inclusion in \eqref{reduced WF}, we note that if
$\xoxio\notin \Vab$, then neither of the two angles
$\po=\po(\xo,\xio)$ and $\pt=\pt(\xo,\xio)$ is in $[a,b]$.  Therefore
$\Kvp\Mc (f)$ is zero (and so smooth) near $\phi_j$ for $j=1,2$ and
for any $\eta$ above $(\phi_j,r)$.  That means that any covector above
the point $(\phi_j,\norm{\xo-\th(\phi_j)})$ ($j=1,2$) is not in
$\WF(P\Kvp\Mc(f))$. By \eqref{cinv}, the only two covectors that could
contribute to $\Cmt\circ \WF(P\Kvp\Mc(f))$ at $\xoxio$ are above these
two points.  Since $\WF(\Lvp(f))\subset \Cmt\circ \WF(P\Kvp\Mc(f))$,
$\xoxio\notin \WF(\Lvp(f))$, this proves the theorem.
\end{proof}

\bigskip
\begin{proof}[Proof of Theorem \ref{thm:visible and added sing
sphere}] We will recall some of the microlocal analysis of the sonar
transform that was proven in \cite{QRS2011}.
The canonical relation of $\Ms$ is \cite[equation (3.5)]{QRS2011} 
\bel{def:C:sphere}
\begin{aligned}\Cs =
\big\{\paren{(y,r),\alpha\bparen{\n'(y,x)\dy+\dr},x,\alpha\n(y,x)\dx}\st\phantom{x}
\\(y,r)\in \Xis,\ x\in S(y,r)\cap X,\ \alpha\neq
0\big\}.\end{aligned}\ee Because $x_3>0$, $\xi=\alpha\n(y,x)$ always
has $\xi_3\neq 0$.  Note that in equation (3.5) in \cite{QRS2011} the
$\dx$ coordinate should be $+\alpha\n(y,x)\dx$ since $C$ is the
canonical relation for $\Ms$, not its Lagrangian manifold.

Let $(x,\xi\dx)\in T^*(X)$ with $\xi_3\neq 0$.  We define
\bel{def:c}c(x,\xi\dx) =
\paren{y(x,\xi),r(x,\xi),\alpha(x,\xi)\bparen{\om'(\xi)\dy+\dr}}\ee
where \bel{def:yr etc.}\begin{array}{rlrl} y(x,\xi) &=
\paren{x-\frac{x_3}{\xi_3}\xi}'\qquad& r(x,\xi) &=
\frac{x_3}{\abs{\xi_3}}\norm{\xi}\\ \alpha(\xi)&=
\frac{\xi_3}{\abs{\xi_3}}\norm{\xi}&\om(\xi) &=
\frac{\xi_3}{\abs{\xi_3}\norm{\xi}}\xi\in S^2_+\end{array}\ee and 
$S^2_+$ is the open upper hemisphere of $S^2$.

Recall that $(x_1,x_2,x_3)' = (x_1,x_2)$.
  Then, by
equations (3.7)-(3.10) in \cite{QRS2011}, \bel{Cs and c}\Cs\circ
\sxxi = \sparen{c(x,\xi\dx)}.\ee

Note that $c$ is a diffeomorphism from $\sparen{(x,\xi)\st x\in X,
\xi_3\neq 0}$ onto
\bel{def:Ts}T_{\Sc}=\sparen{\paren{(y,r),\alpha[\eta\dy+\dr]}\st(y,r)\in
\Xis,\ \alpha\neq 0,\ \eta\in D}\subset T^*(\Xis)\ee where $D$ is the
open unit disk in $\rtwo$.  

Let $((y,r),\nu\dy+\alpha\dr))\in T_\Sc$, Then, $\nu/\alpha\in D$ and
so \bel{c inverse}\begin{gathered}\vn(\nu,\alpha) =
\paren{\nu/\alpha,\sqrt{1-\norm{\nu/\alpha}^2}}\in
S^2_+\ \text{ and }\ x(y,r,\nu,\alpha)=(y,0)+r\, \vn(\nu,\alpha)\\
\text{satisfy}\\ c^{-1}(((y,r),\nu\dy+\alpha\dr))=
(x(y,r,\nu,\alpha),\alpha \vn(\nu,\alpha)\dx).\end{gathered}\ee Because
of relation \eqref{c inverse} and the definition of $\Cst$, if
$((y,r),\nu\dy+\alpha\dr)\in T_{\Sc}$ then \[\Cst
\circ\sparen{((y,r),\nu\dy+\alpha\dr)} =
\sparen{c^{-1}\paren{(y,r),\nu\dy+\alpha\dr}}.\] Therefore,
\bel{Delta:sphere}\Cst \circ \Cs = \sparen{(x,\xi\dx;x,\xi\dx)\st x\in
X,\ \xi_3\neq 0}\ee is a dense open subset of the diagonal
$\Delta\subset T^*(X\times X)$.  This finishes the general microlocal
analysis of $\Ms$ and $\Msst$.

Let $f\in \Ec'(X)$ and let $K$ be a compact set in the plane with
nontrivial interior and let $A=K\times (0,\infty)$.

Because every covector in $T_\Sc$ has nonzero $\dr$ component and
every covector in $\WF(\chi_A)$ has zero $\dr$ component, the
non-cancellation condition \eqref{non-cancellation} holds. So, by
Theorem \ref{thm:WF mult} the product $\chi_A \Ms f=\Ma (f)$ is a
distribution.  By Lemma \ref{lemma:composition identity}, $\WF(\Ma
(f))\subset\Qc(A, \Cs\circ\WF(f))$.  As with the PAT case,
$\Qc(A,\Cs\circ\WF(f))$ breaks into the union of three sets
\begin{align}
\label{eq:wavefront set restricted Radon transform:sphere} \Qc(A,
\Cs\circ \WF(f))&= \bparen{\Cs\circ\WF( f)\cap \sparen{(y,r),\eta)\in T^*(\Xis)\st
y\in K}}\\
& \cup \WF(\chi_{A}) \cup \Wbd(f) \end{align}
where \bel{def:Wbd}\begin{aligned} \Wbd(f) = \Big\{ &\paren{(y,r),
\eta\dy +
\alpha\bparen{\n'(y,\tx)\dy + \dr}}
\st (y,\eta)\in N(K), \\ &\quad
 \alpha\neq 0,\  \tx\in
S(y,r),  \text{ and }
\paren{\tx,\alpha\n(y,\tx)\dx}\in\WF(f)\Big\}
\end{aligned}\ee  corresponds to the sum of  covectors
$((y,r),\eta\dy)\in \WF(\chi_A)$ and covectors 
\[((y,r),\alpha\bparen{\n'(y,\tx)\dy + \dr}) \in   \WF(\Cs\circ
\WF(f)).\]

We now examine the fibers of $\Wbd(f)$.  Fix $r>0$, and $\alphao\neq
0$ and $(y,\etao)\in N(K)$.  Assume there is a $\tx\in S(y,r)$ with
$(\tx,\alphao\n(y,\tx)\dx)\in \WF(f)$.  Note that \bel{WF add}\forall
a\neq 0,\ ((y,r),a\etao\dy)\in \WF(\chi_{A}), \quad\forall b>0,
\paren{(y,r),b\alphao\bparen{\n'(y,\tx)\dy+\dr}}\in \Cs\circ\WF(f).\ee
The first statement in \eqref{WF add} follows because $\chi_A$ is a
real valued function so $a$ can be negative as well as positive (for
real $f$ and any real cutoff function $\vp$, $\widehat{\vp f}(\xi)$ is
the complex conjugate of $\widehat{\vp f}(-\xi)$).  The right-hand
statement in \eqref{WF add} is true since $\WF(f)$ is conic.
Therefore the part of the fiber of $\Wbd(f)$ above $(y,r)$ that comes
from a singularity of $f$ at $(\tx,\alphao\n(y,\tx))$ consists of
points \bel{def:Wbdfiber}
\paren{(y,r),\bparen{a\etao+b\alphao\n'(y,\tx)}\dy+b\alphao\dr)}
\text{ for $a\in \rr,\ b>0$}.\ee At this point, we have finished steps
\eqref{paradigm:start} to \eqref{paradigm:Mf} of the paradigm of
Section \ref{sect:general strategy}.

The rest of the proof of \eqref{spherical mean singularities} is
similar to the proof of Theorem \ref{thm:visible and added sing}.
Composing $\Qc(A,\Cs\circ \WF(f))$ with $\Cst$, one sees that
$\Cst\circ \Qc(A,\Cs\circ\WF(f))$ is the union of three sets.  The
first set, from containment \eqref{eq:wavefront set restricted Radon
transform:sphere} becomes $\Cst\circ \bparen{\Cs\circ\WF(f)\cap
\sparen{(y,r),\eta)\in T^*(\Xis)\st y\in K}} = \WF(f)\cap \Vsk$.  The
second composition is $\Cst\circ\WF(\chi_A) = \emptyset$ since
$c^{-1}$ is not defined unless the $\dr$ component is not zero.  The
third composition is $\Cst\circ \Wbd(f)$ and this simplifies to
$\As{K}(f)$ if one uses \eqref{c inverse} and the expression for
part of the fiber of $\Wbd(f)$ given in \eqref{def:Wbdfiber}.

Finally, one applies Lemma \ref{lemma:composition identity},
\eqref{composition-genl} to finish the proof of \eqref{spherical mean
singularities}.

To prove the second assertion in Theorem \ref{thm:visible and added
sing sphere} if $P$ is elliptic, we use the result of Theorem
\ref{thm:reduction:sphere}, \eqref{elliptic spherical}, which is proven
independently of this part of the proof.  Let $(x,\xi\dx)\in
\WF(f)\cap \Vs{\intt(K)}$ and let $y=y(x,\xi)$ be the point in
$\intt(K)$ that is the center of the sphere $S(y(x,\xi), r(x,\xi))$
containing $x$ and which is normal to $\xi$ at $x$. Now, let $\vp$ be
a cutoff function that is supported in $\intt(K)$ and is equal to one
in a compact set $K'\subset \intt(K)$ chosen so $y\in \intt(K')$.
Then, by Theorem \ref{thm:reduction:sphere}, $(x,\xi\dx)\in
\WF(\Lvp(f))$.  However, 
\[(x,\xi\dx)\notin \WF(\Msst P (\chi_A-\vp)\Ms(f))\] for the following
reasons.  First, $(\chi_A-\vp)$ is zero near $(y,r)$ for all $r$. Thus
$P(\chi_A-\vp)\Ms(f)$ is smooth near $y$ and so $\Msst
P(\chi_A-\vp)\Ms(f)$ is smooth near $c\inv((y,r),\nu\dy+\alpha\dr)$
whenever $\nu/\alpha\in D$.  However, for some choice of $r, \nu$, and
$\alpha$, $((y,r),\nu\dy+\alpha\dr) = c(x,\xi\dx)$ and so $(x,\xi\dx)
= c^{-1}((y,r),\nu\dy+\alpha\dr)$.  Thus $\Msst P(\chi_A-\vp)\Ms (f)$
is smooth near $(x,\xi\dx)$.  These two results show that
$(x,\xi\dx)\in \WF(\Msst P\Msa(f))$ and this proves the second part of
the theorem.
\end{proof}

\bigskip
\begin{proof}[Proof of Theorem \ref{thm:reduction:sphere}]
Let $\vp$ be a cutoff function supported in $\intt(K)$ and equal to
one on the set $K'$ given in the statement of this theorem.  Then,
$\Kvp$ is trivially a pseudodifferential operator on $\Ec'(\Xis)$ and
so $P\Kvp\Ms$ is a FIO associated to $\Cs$.  Now, by
\eqref{Delta:sphere}, $\Lvp$ is a standard pseudodifferential
operator, and this proves the first part of Theorem
\ref{thm:reduction:sphere}.

To prove the ellipticity statement, \eqref{elliptic spherical}, let
$(x,\xi\dx)\in \WF(f)\cap \Vs{K'}$ and assume $P$ is elliptic. Let
$\nu\in T^*_{(y,r)}(\Xi)$ and $((y,r),\nu)=c(x,\xi\dx)$.  Because
$(x,\xi\dx)\in \Vs{K'}$, $y\in K'$.  Because $\Ms$ is elliptic (and
$\Pi_L:C\to T^*(\Xi)$ is an injective immersion) and $P\Kvp$ is
elliptic near $((y,r),\nu)$, one sees that $((y,r),\nu)\in
\WF(P\Kvp\Ms(f))$. Since $\Msst$ is elliptic and both $c$ and $c\inv$
are functions, $c\inv((y,r),\nu)=(x,\xi\dx)$ and $(x,\xi\dx)\in
\WF(\Lvp(f))$. This proves the second part of the theorem.
\end{proof}

\bigskip
\begin{remark}\label{remark:elliptic}
In the case of sonar, we can assert, if $P$ is elliptic, that $\Msst P
\Kvp\Ma$ is elliptic on $\Vs{\intt(K')}$ because composition with
$\Cs$ and $\Cst$ is described by two functions, $c$ and $c\inv$.  This
reflects the fact that for each $(x,\xi\dx)\in \Vs{\rtwo}$ there is a
unique $((y,r),\nu)$ in $\Cs\circ \sparen{(x,\xi\dx)}$.  In the case
of the circular transform in PAT, there are two such covectors,
corresponding to the two angles $\po(x,\xi)$ and $\pt(x,\xi)$ in the
proof of Theorem \ref{thm:visible and added sing}.  For the circular
transform, to assert that $\Mst P \Kvp\Mc$ is elliptic in
$\Vc_{(a',b')}$, one would have to calculate the symbol of $\Mst P\Kvp
\Mc$ and make sure it is nowhere zero on $\Vc_{(a',b')}$.  

This argument implies that $\Mst \Kvp\Mc$ is elliptic on
$\Vc_{(a',b')}$ since the symbol consists of terms that are positive
above $\Vc_{(a',b')}$ (see e.g., \cite{Gu1975, GS1977, Q1980}).  The
same argument shows that $\Mst P \Mc$ would be elliptic on
$\Vc_{(a',b')}$ when $P$ is a pseudodifferential operator with
top-order symbol that is everywhere positive (or everywhere negative)
above $(a',b')\times (0,\infty)$.

As a simple example where this is not true, if $R$ is the Radon line
transform and $R^*$ is its dual, then $R^* (d/dp) R = 0$ even though
each operator is elliptic ($d/dp$ is elliptic on range$(R)$).
\end{remark}

\begin{remark}\label{remark:subtle} Our analysis is simplified
because $C^t\circ C \subset \Delta$ in \eqref{Delta:circle} and
\eqref{Delta:sphere}.  Because of this, the set of ``visible''
singularities (\eqref{WF(f)} for PAT and $C^t_S $ composed with the
first set in \eqref{eq:wavefront set restricted Radon
transform:sphere} for sonar) is a subset of $\WF(f)$ since $C^t\circ
\paren{C\circ\WF(f)}\subset \Delta\circ \WF(f) = \WF(f)$.

If $C$ is the canonical relation of a FIO $\Fc$ and $\Pi_L:C\to
T^*(\Xi)$ is an injective immersion, then $C^t\circ C \subset \Delta$
and if they can be composed, $\Fc^* \Fc$ is a pseudodifferential
operator \cite{Gu1975, GS1977, Q1980}.

When $C^t\circ C$ is not a subset of $\Delta$, singularities can be
added by the backprojection itself, even with a smooth cutoff in place
of $\chi_A$.  This comes up, for example, in common midpoint synthetic
aperture radar: the normal operator $\Mc^* \Mc$ is not even a FIO but
a sum of singular FIOs associated with different canonical relations
\cite{AFKNQ:common-midpoint}.
\end{remark}

\end{appendix}

\bibliographystyle{siam}
\bibliography{references}

\end{document}